\documentclass{amsart}

\usepackage[scale=0.78]{geometry}

\usepackage{amsmath,amsbsy,amsfonts,amsopn,amstext, bm, amsthm}
\usepackage{euscript,amssymb,amsbsy,amsfonts,amsthm,latexsym,amsopn,amstext,amsxtra,euscript,amscd}
\usepackage{array,arydshln}
\usepackage{graphicx, verbatim}

\usepackage{longtable}

\usepackage{hyperref}

\hypersetup{
  colorlinks   = true, 
  urlcolor     = blue, 
  linkcolor    = blue, 
  citecolor   = red 
}

\usepackage[capitalise]{cleveref}

\newtheorem*{thm*}{Theorem} 
\newtheorem{thm}{Theorem}

\newtheorem{prop}{Proposition} 
\newtheorem{lem}{Lemma}

\newtheorem{exa}{Example} 

\newtheorem{rem}{Remark}
\newtheorem{cor}{Corollary}

\newtheorem{clm}{Claim}

\crefname{thm}{Thm.}{}
\crefname{prop}{Prop.}{}
\crefname{lem}{Lem.}{}
\crefname{cor}{Cor.}{}

\newcommand\Z{\mathbb Z}
\newcommand\Q{\mathbb Q}
\newcommand\R{\mathbb R}

\def\mod{\mbox{ mod }}
\def\deg{\mbox{deg }}

\def\sym{\mbox{Sym}}

\def\<{\langle}

\usepackage[msc-links]{amsrefs}

\title[On the discriminant of a certain quadrinomials]{Bezoutians and the discriminant of a certain quadrinomials}

\author{Shuichi Otake}
\address{Department of Applied Mathematics\\ Waseda University \\ Japan }
\email{shuichi.otake.8655@gmail.com}

\author{Tony Shaska}
\address{Department of Mathematics and Statistics \\ Oakland  University \\ Rochester, MI, 48309. }
\email{shaska@oakland.edu}

\date{}                                           

\def\Gal{\mbox{Gal }}

\begin{document}

\begin{abstract}
We give an explicit  formula for the discriminant $\Delta_f (x)$ of the  quadrinomials of the form 
\[f (x)=x^n+ t (x^2+ax+b)\] 
The proof uses  Bezoutians of polynomials.  This paper is an extended version of \cite{o-sh-3}.
\end{abstract}

\maketitle

\section{Introduction}\label{intro}

In \cite{Se} Selmer studied polynomials $f(x)=x^n\pm x \pm 1$ and proved that $x^n-x-1$ is irreducible for all $\geq 2$, while $x^n+x+1$ is irreducible for $n \not\equiv 2 \mod 3$.   In the process he proved that the discriminant of $f(x)$ is (up to a multiplication by $-1$) 
$ \Delta = \pm \left( n^n \pm (n-1)^{n-1} \right),$
see \cite{Se} for precise formulas. He also noticed that polynomials $f(x)$ have very small discriminants. He checked his results for $n \leq 20$.   Other mathematicians have considered trinomials $f(x) = x^n + ax+b$ for various reasons; see \cite{Mori}, \cite{Za}. 
In general, finding conditions on the coefficients of a polynomial to have minimal discriminant is a difficult problem.  It is related to reduction of binary forms discussed in \cite{reduction} and their heights \cite{height}. More generally it is a special case of  finding conditions on the coefficients of a binary form such that the corresponding  point in the weighted moduli space of invariants is normalized; see \cite{w-height}. 

This paper is the first of hopefully others to come to determine for what $a, b, t$ the quadrinomial 
\begin{equation}\label{eq-0}
f(x) = x^n + t  (x^2+ax +b) 
\end{equation} 
has minimal discriminant, is reduced in the sense of \cite{reduction}, or  has minimal naive height.   There have been plenty of efforts to determine such formulas for certain classes of polynomials. In \cite{g-d}, \cite{swa}  the authors focus on  the computation of the discriminant of a trinomial which has been carried out in different ways. In this paper, we determine explicitly a formula for  the discriminant of $f(x)$ by using the approach of Bezoutians. 

We prove (see \cref{thm6.0}) that the discriminant $\Delta_f(x)$ of the polynomials in Eq.~\eqref{eq-0} is  given by the formula
\begin{equation}
\Delta =(-1)^{m_1}t^{n-1}   \left( (n-2)^{n-2}(a^2-4b)t^2+\gamma_{\bm c} \, t-n^nb^{n-1} \right), 
\end{equation}
such that  
\[
\gamma_{\bm c} =  \displaystyle \sum_{k=0}^{m_0}(-1)^{n+k}n^k(n-1)^{n-2k-4}(n-2)^ka^{n-2k-4} b^k \cdot S_k,  
%
\]
and $m_0=\lfloor (n-3)/2 \rfloor$, $m_1=\lceil (n-3)/2 \rceil$  for $a \neq 0$  or  $a=0$ and $n$ even,  and $\gamma_{\bm c} =0$ for $a=0$ and $n$ odd, and
\begin{small}
\begin{align*}
S_k=&(n-1)^3\binom{n-k-3}{k}a^4-\frac{n(n-1)\left\{ 5n^2-(6k+23)n+10k+24 \right\}}{n-k-3}\binom{n-k-3}{k}a^2b +4n^2(n-2)\binom{n-k-4}{k}b^2. 
\end{align*}
\end{small}

It is object of further investigation if such result could be generalized to polynomials $f(x)= x^n+ t \cdot g(x)$, where $g(x)$ is a general cubic.  As a quick application of \cref{thm6.0} we get that   for $b\neq 0$ and $\frac {(n-1)^2 a^2 } {4n (n-2)} \leq b$ the polynomial $f (t, x)$ has no real roots, for any real number $t >0$;  \cref{thm6.1}.   While there is an elementary proof of \cref{thm6.1}, it is interesting to check whether out approach would work for any polynomial of type $f(t, x) = x^n + t \cdot g(x)$, for $\deg g \geq 3$. 

It is a quite open problem to determine for what integer values of $a$, $b$, and  $t$ the quadrinomial $f(x)$ is irreducible or the discriminant $\Delta_f$ has a minimal value.  Moreover, it is our intention to study in the future for what conditions on $a, b, t$ the quadrinomial is reduced in the sense of \cite{reduction}.

Another motivation of looking at the  discriminant of such family of polynomials comes from our efforts to construct superelliptic Jacobians with large endomorphism rings.  We have checked computationally that for small $n$ and  $f(t, x) \in \Q[t, x]$, for almost all $t$ the Galois group $\Gal_{\Q} (f, x)$ is isomorphic to $S_n$.  Due to results of Zarhin this implies that the superelliptic curve $y^m = f(t, x)$ has large endomorphism ring; see \cite{f-sh} for related matters. Hence, curves $y^n=f(x)$, where $f(x)$   is as above, are smooth curves whose  Jacobians are expected to have large endomorphism rings. This remains the focus of further investigation.

\bigskip

\noindent \textbf{Notation:} Throughout this paper $n$ is a positive integer,  $\Delta$   denotes the discriminant of a polynomial $f(x)$ with respect to the variable $x$. 
The symbol   $\lfloor \cdot \rfloor$ is the floor function and $\lceil \cdot \rceil$ is the ceiling function.  $B_n (f, g)$ denotes the level $n$ Bezoutian of two polynomials $f$ and $g$. For a polynomial $f \in \R[x]$ the symbol $N_f$ denotes the number of real roots of $f(x)$.  We denote by $\binom{n}{m}$ $(n,m \in \Z_{\geq 0})$ the binomial coefficient. 

\section{Preliminaries}
Let $F$ be a field of characteristic zero and $f_1(x)$, $f_2(x)$ be polynomials over $F$. Then, for any integer $n$ such that  $n \geq \max \{ \mathrm{deg} f_1, \mathrm{deg} f_2 \}$, we put
\begin{align*}
B_{n}(f_{1},f_{2}) :&= \frac{f_{1}(x)f_{2}(y)-f_{1}(y)f_{2}(x)}{x-y}= \sum_{i,j=1}^{n} \alpha _{ij}x^{n-i}y^{n-j}\in F[x,y], \\ 
M_{n}(f_{1},f_{2}) :&= (\alpha_{ij})_{1 \leq i,j \leq n}.
\end{align*} 
The $n \times n$ matrix $M_{n}(f_{1},f_{2})$ is called the Bezoutian of $f_1$ and $f_2$. When $f_2=f_1^{\prime}$, the formal derivative of $f_1$ with respect to the indeterminate $x$, we often write $B_n(f_1):=B_{n}(f_{1},f_{1}^{\prime})$, $M_n(f_1):=M_n(f_1, f_1^{\prime})$ and the matrix $M_n(f_1)$ is called the Bezoutian of $f_1$. We denote by $N_{f_1}$ the number of distinct real roots of $f_1(x)$ and $\Delta(f_1)$ the discriminant of $f_1(x)$. Moreover, for any real  symmetric matrix $M$, we denote by $\sigma(M)$ the index of inertia of $M$. The following are the list of important properties of Bezoutians.

\begin{lem}\label{prop2.1}
Notations as above, we have
\begin{enumerate}
\item[$(i)$] $M_n(f_1, f_2)$ is an $n \times n$ symmetric matrix over $F$. $(M_n(f_1, f_2) \in \sym_n(F))$.
\item[$(ii)$] $B_{n}(f_{1},f_{2})$ $(M_{n}(f_{1},f_{2}))$ is linear in $f_{1}$ and $f_{2}$, separately. 
\item[$(iii)$] $B_{n}(f_{1},f_{2})=-B_{n}(f_{2},f_{1})$ $(M_{n}(f_{1},f_{2})=-M_{n}(f_{2},f_{1}))$.
\item[$(iv)$] $N_{f_1}=\sigma(M_n(f_1))$.
\item[$(v)$] $\displaystyle \Delta(f_1)=\frac{1}{led(f_1)^2}\det M_n(f_1)$, where $led(f_1)$ is the leading coefficient of $f_1$. In particular, if $f_1$ is monic, we have $\displaystyle \Delta(f_1)=\det M_n(f_1)$.
\item[$(vi)$] Let $\lambda ,\mu ,\nu$ be integers such that $\lambda \geq \mu > \nu \geq 0$. Then
$M_{\lambda }(x^{\mu }, x^{\nu }) = (m_{ij} )_{1\leq i,j\leq \lambda }$,
where
\begin{equation*}
m_{ij}=\begin{cases}
               1 & \text{$i+j=2\lambda-(\mu+\nu)+1$ \ $(\lambda-\mu+1\leq i,j\leq \lambda-\nu )$}, \\
               0 & \text{otherwise}.
             \end{cases}
\end{equation*}
\end{enumerate}
\end{lem}

\begin{proof}
For (i)--(iv), see \cite[Theorem 8.25, 9.2]{fuh}. For (v), see equation (4) of \cite[p.217]{fro}. For (vi), see \cite[Lemma 2]{o-sh}
\end{proof}
For any vector ${\bm r}=(r_0, \cdots, r_s) \in \R^{s+1}$, let us put 
\begin{align*}
g_{\bm r}(x)=g(r_0,\cdots, r_s ; x) = \sum_{k=0}^{s}r_{s-k}x^{s-k} \in \R[x], \\ 
f_{\bm r}(t; x)=f^{(n)}(r_0,\cdots, r_s, t ; x)=x^n+t \cdot g_{\bm r}(x) \in \R(t)[x].
\end{align*}

In the previous paper \cite{o-sh}, by using the above properties of Bezoutians, we obtained the next theorem

\begin{thm}{\cite[Thm.~2]{o-sh}} \label{thm5.2}
Let ${\bm r}=(r_0, \cdots, r_s) \in \R^{s+1}$ be a vector such that $g_{\bm r}(x)$ is a degree $s$ separable polynomial satisfying $N_{g_{\bm r}(x)}=\gamma$ $(0 \leq \gamma \leq s)$. Let us consider $f_{\bm r}(t ; x)=f^{(n)}(r_0, \cdots, r_s, t ; x)$ as a polynomial over $\R(t)$ in $x$ and put 
\[ 
P_{\bm r}(t)=\det M_n(f_{\bm r}(t ; x))=\det M_n(f_{\bm r}(t ; x), f_{\bm r}^{\prime}(t ; x)),
\] 
where $f_{\bm r}^{\prime}(t ; x)$ is a derivative of $f_{\bm r}(t ; x)$ with respect to $x$. Then, for any real number $t>\alpha_{\bm r}=\max\{ \alpha \in \R \mid P_{\bm r}(\alpha)=0 \}$, we have 
\begin{align*}
N_{f_{\bm r}(t ; x)}
=
\begin{cases}
\gamma+1 & \text{$n-s$ $:$ odd} \\
\gamma & \text{$n-s$ $:$ even, $r_s>0$} \\
\gamma+2 & \text{$n-s$ $:$ even, $r_s<0$}.
\end{cases}
\end{align*}
\end{thm}

Then, by using \cref{thm5.2} and  \cref{thm6.0}, we construct a certain family of totally complex polynomials of the form $f_{(b,a,1)}(t; x)$; see \cref{thm6.1}. Here, note that a real polynomial $f(x)$ to be called totally complex if it has no real roots, that is, $N_f=0$. 


Let $n \geq 3$ be an integer and $a$, $b$ be real numbers.   Let us put ${\bm c}=(b,a,1) \in \R^3$ and
\begin{equation*}
g_{\bm c}(x)=x^2+ax+b \in \R[x], \ f_{\bm c}(t; x)=x^n+tg_{\bm c}(x) \in \R(t)[x].
\end{equation*}
We denote by $\binom{n}{m}$ $(n,m \in \Z_{\geq 0})$ the binomial coefficient. Note that $\binom{n}{0}=1$ for any integer $n \geq 0$ and $\binom{n}{m}=0$ if $n<m$. Moreover, we define $\binom{n}{m}=0$ for any $n \in \Z_{\geq 0}$ and $m \in \Z_{<0}$.

Our main result is the following. 

\begin{thm} \label{thm6.0}
Put $m_0=\lfloor (n-3)/2 \rfloor$, $m_1=\lceil (n-3)/2 \rceil$. Moreover, for any integers $n\geq 4$ and $k$ $\left(0 \leq k \leq m_0\right)$, put 
\begin{small}
\begin{align*}
S_k=&(n-1)^3\binom{n-k-3}{k}a^4-\frac{n(n-1)\left\{ 5n^2-(6k+23)n+10k+24 \right\}}{n-k-3}\binom{n-k-3}{k}a^2b +4n^2(n-2)\binom{n-k-4}{k}b^2. 
\end{align*}
\end{small}
Here, $\lfloor \cdot \rfloor$ is the floor function and $\lceil \cdot \rceil$ is the ceiling function. \\[1mm]
$(1)$ \ Suppose $n=3$. Then, we have
\begin{align*}
\Delta\left(f_{\bm c}(t; x)\right)=t^2\Biggl\{(a^2-4b)t^2-(4a^3-18ab)t-27b^2 \Biggr\}.
\end{align*}
$(2)$ \ Suppose $n\geq 4$. Then, we have
\begin{align*}
\Delta\left(f_{\bm c}(t; x)\right)=(-1)^{m_1}t^{n-1}\Biggl\{(n-2)^{n-2}(a^2-4b)t^2+\gamma_{\bm c}t-n^nb^{n-1} \Biggr\}, 
\end{align*}
where
\begin{align*}
\gamma_{\bm c}=
\begin{cases}
\displaystyle \sum_{k=0}^{m_0}(-1)^{n+k}n^k(n-1)^{n-2k-4}(n-2)^ka^{n-2k-4}b^kS_k & \hspace{-1mm} \text{$(a \neq 0$ or $a=0$, $n:$ even$)$}, \vspace{2mm} \\
0 & \hspace{-1mm} \text{$(a=0$, $n:$ odd$)$}.
\end{cases}
\end{align*}
\end{thm}

As a corollary of the above we get.

\begin{cor} \label{thm6.1}
Suppose $n\geq 4$ is an even integer. Moreover, let us suppose $b \neq 0$ and $(n-1)^2a^2/4n(n-2)\leq b$. Then we have $N_{f_{\bm c}(t; x)}=0$ for any positive real number $t$.
\end{cor}

By a direct computation, we have
\begin{align*}
&\Delta\left(x^3+t(x^2+ax+b)\right)=t^2\left\{ (a^2-4b)t^2-(4a^3-18ab)t-27b^2 \right\}, \\
&\Delta\left(x^4+t(x^2+ax+b)\right)=-t^3\left\{ (4a^2-16b)t^2+(27a^4-144a^2b+128b^2)t-256b^3 \right\}
\end{align*}
and we get \cref{thm6.0} for $n=3$, $4$. Moreover, if $n=4$ and $(n-1)^2a^2/4n(n-2)=b$, we have
\begin{align*}
P_{\bm c}(t)=\Delta\left(x^4+t(x^2+ax+b)\right)=\frac{1}{2}a^2\left(t+\frac{27}{8}a^2 \right)^2t^3
\end{align*}
and hence $\alpha_{\bm c}=0$, which implies $N_{f_{\bm c}(t; x)}=0$ for any $t>0$ by \cref{thm5.2} since $x^2+ax+b$ has no real root in this case. Then, by considering the graph of the function $y=x^4+t(x^2+ax+b)$, we have $N_{f_{\bm c}(t; x)}=0$ whenever $(n-1)^2a^2/4n(n-2)\leq b$ and $t>0$, which is the claim of \cref{thm6.1} for $n=4$. Therefore, let us assume $n \geq 5$ hereafter.

In the following, we put 
\begin{align*}
&A_{\bm c}(t)=(a_{ij}^{(\bm c)}(t))_{1 \leq i, j \leq n}=M_n(f_{\bm c}(t; x)) \in \mathrm{Sym}_n(\R(t)), \\
&B_{\bm c}=(b_{ij}^{(\bm c)})_{1 \leq i,j \leq 2}=M_2(g_{\bm c}(x)) \in \mathrm{Sym}_{2}(\R), \\
&P_{\bm c}(t)=\det A_{\bm c}(t)=\Delta(f_{\bm c}(t; x)). 
\end{align*}
Then, by   \cref{thm5.2}, we have the next lemma.
\begin{cor} \label{lem6.1}
Suppose $n$ is even and $a^2/4<b$. Put $\alpha_{\bm c}=\max \{ \alpha \in \R \mid P_{\bm c}(\alpha)=0 \}$. Then, for any real number $t>\alpha_{\bm c}$, we have $N_{f_{\bm c}(t, x)}=0$.
\end{cor} 

To prove \cref{thm6.0}, we will use the equality $\Delta(f_{\bm c}(t; x))=\det A_{\bm c}(t)$. Also, to prove \cref{thm6.1}, it is enough to prove $\alpha_{\bm c}=0$ when $(n-1)^2a^2/4n(n-2)\leq b$ and $(a,b)\neq(0,0)$, which implies we need to compute the polynomial $P_{\bm c}(t)=\det A_{\bm c}(t)$ precisely. Here, let $Q_{m}(k;c)=(q _{ij})_{1\leq i,j\leq m}$,  $R_{m}(k,l;c)=(r _{ij})_{1\leq i,j\leq m}$ be $m\times m$ elementary matrices such that
\begin{center}
\noindent $Q_m(k;c)$=
\scalebox{0.85}[1]
{$\left[
\begin{array}{ccccccc} 
1 & & & & & & \\ 
& \ddots & & & & & \\ 
& & 1 & & & & \\ 
& & & c & & & \\ 
& & & & 1 & & \\ 
& & & & & \ddots & \\ 
& & & & & & 1
\end{array}
\right]
$},
$R_m(k,l;c)$=
\scalebox{0.85}[1]
{$\left[
\begin{array}{cccccccc} 
1 & & & & & & & \\ 
& \ddots & & & & & & \\ 
& & 1 & & & c & \\ 
& & & \ddots & & & & \\ 
& & & & & 1 & & \\ 
& & & & & & \ddots & \\ 
& & & & & & & 1
\end{array}
\right],
$}
\end{center}
where $q _{kk}=c$, $r _{kl}=c$, respectively and, for any $m \times m$ matrices $M_1$, $M_2$, $\cdots$, $M_l$, put
$\prod_{k=1}^{l} M_k=M_1M_2\cdots M_l$. Moreover, put $l_k=n+k$. Then, as is the case in \cite{o-sh}, we inductively define the matrix $A_{\bm c}(t)_k$ $(1 \leq k\leq n-2)$ as follows:

\begin{itemize}
\item[(1)] $A_{\bm c}(t)_1=(a_{ij}^{(\bm c)}(t)_1)_{1 \leq i,j \leq n}={}^tS_{\bm c}(t)_1A_{\bm c}(t)S_{\bm c}(t)_1$, where 
\[
S_{\bm c}(t)_1=\displaystyle Q_{n}(1;1/\sqrt{n})\prod_{k=0}^{1}R_n(1,l_k-1;-a_{1, l_k-1}^{(\bm c)}(t)/\sqrt{n}). 
\]
\item[(2)] Put
\[
n_0 =
\begin{cases}
(n-1)/2, & \text{$n$ $:$ odd}, \\
n/2, & \text{$n$ $:$ even}.
\end{cases}
\]
Then, $A_{\bm c}(t)_{k}=(a_{ij}^{(\bm c)}(t)_k)_{1 \leq i,j \leq n}={}^tS_{\bm c}(t)_kA_{\bm c}(t)_{k-1}S_{\bm c}(t)_k$, where
\[
S_{\bm c}(t)_k=
\begin{cases}
\displaystyle \prod_{m=l_0-k+1}^{n}R_n\left(l_0-k,m;-\dfrac{a_{km}^{({\bm c})}(t)_{k-1}}{-(n-2)t}\right), &     \hspace{-24mm}\text{($2 \leq k \leq n_0$)}  \\[3mm]
\displaystyle R_n\left(l_0-k,k;-\dfrac{a_{kk}^{({\bm c})}(t)_{k-1}}{-2(n-2)t}\right)\prod_{m=k+1}^{n}R_n\left(l_0-k,m;-\dfrac{a_{km}^{({\bm c})}(t)_{k-1}}{-(n-2)t}\right), \; \;  \text{($n_0 < k \leq n-2$)}.
\end{cases}
\] 
\end{itemize}

By the definition of $A_{\bm c}(t)_k$ $(1 \leq k\leq n-2)$, we have $\det A_{\bm c}(t)=n\det A_{\bm c}(t)_{n-2}$. Therefore, to compute the polynomial $P_{\bm c}(t)=\det A_{\bm c}(t)$, let us first compute the matrix $A_{\bm c}(t)_{n-2}$ concretely. 

\section{Computation of the matrix $A_{\mbox{\scriptsize \boldmath $c$}}(t)_{n-2}$}
By \cref{prop2.1}, we have 
\[
B_{\bm c} =M_2(x^2+ax+b, 2x+a) =2M_2(x^2, x)+aM_2(x^2, 1)+(a^2-2b)M_2(x,1) =
\left[
\begin{array}{cc}
2 & a \\
a & a^2-2b  
\end{array}
\right].
\]
Here, let us give some examples of the matrices $A_{\bm c}(t)$ and $A_{\bm c}(t)_1$ for some small $n$.

\begin{exa} \label{exa6.1}

\begin{enumerate}

\item
Let $n=5$. Then, 
\begin{align*}
A_{\bm c}(t)&=
\scalebox{0.9}[0.9]
{$\left[
\begin{array}{ccccc}
5 & 0 & 0 & 2t & at \\
0 & 0 & -3t & -4at & -5bt \\
0 & -3t & -4at & -5bt & 0 \\
2t & -4at & -5bt & 2t^2 & at^2 \\
at & -5bt & 0 & at^2 & (a^2-2b)t^2
\end{array}
\right]
$}, \quad 
A_{\bm c}(t)_1&=
\scalebox{0.9}[0.9]
{$\left[
\begin{array}{ccccc}
1 & 0 & 0 & 0 & 0 \\
0 & 0 & -3t & -4at & -5bt \\
0 & -3t & -4at & -5bt & 0 \\
0 & -4at & -5bt & (6/5)t^2 & (3/5)at^2 \\
0 & -5bt & 0 & (3/5)at^2 & \left\{(4/5)a^2-2b\right\}t^2
\end{array}
\right]
$}.
\end{align*}

\item
Let $n=6$. Then,
\begin{small}
\[
A_{\bm c}(t) =
\scalebox{0.9}[0.9]
{$\left[
\begin{array}{cccccc}
6 & 0 & 0 & 0 & 2t & at \\
0 & 0 & 0 & -4t & -5at & -6bt \\
0 & 0 & -4t & -5at & -6bt & 0 \\
0 & -4t & -5at & -6bt & 0 & 0 \\
2t & -5at & -6bt & 0 & 2t^2 & at^2 \\
at & -6bt & 0 & 0 & at^2 & (a^2-2b)t^2
\end{array}
\right]
$}, \;
A_{\bm c}(t)_1 =
\scalebox{0.9}[0.9]
{$\left[
\begin{array}{cccccc}
1 & 0 & 0 & 0 & 0 & 0 \\
0 & 0 & 0 & -4t & -5at & -6bt \\
0 & 0 & -4t & -5at & -6bt & 0 \\
0 & -4t & -5at & -6bt & 0 & 0 \\
0 & -5at & -6bt & 0 & \frac 4 3 t^2 & \frac 2 3 at^2 \\
0 & -6bt & 0 & 0 & \frac 2 3 at^2 & \left( \frac 5 6 a^2-2b \right) t^2
\end{array}
\right]
$}.
\]
\end{small}
\end{enumerate}
\qed
\end{exa}

Let's continue our discussion for  $n\geq 5$. Then, by \cite[Prop~4]{o-sh} and  \cite[Eq.~(5)]{o-sh}, we have
\begin{align} \label{eq6.1}
A_{\bm c}(t)
=
\scalebox{0.81}[0.81]
{$\left[
\begin{array}{cccccccc}
n & 0 & \dots & \dots & \dots & 0 & 2t & at \\
0 &&&&& -(n-2)t & -(n-1)at & -nbt \\
\vdots &&&&  \rotatebox[origin=c]{292}{\vdots} &  -(n-1)at & -nbt & 0 \\
\vdots &&& \rotatebox[origin=c]{292}{\vdots} & \hspace{3mm}  \rotatebox[origin=c]{292}{\vdots} & -nbt & 0 & \vdots \\ 
\vdots && \rotatebox[origin=c]{292}{\vdots} & \hspace{3mm}  \rotatebox[origin=c]{292}{\vdots} & \hspace{3mm}  \rotatebox[origin=c]{292}{\vdots} & \hspace{3mm}  \rotatebox[origin=c]{292}{\vdots} & \vdots & \vdots \\
0 & -(n-2)t & -(n-1)at & -nbt & \hspace{3mm}  \rotatebox[origin=c]{292}{\vdots} && 0 & 0 \\ 
2t & -(n-1)at & -nbt & 0 & \dots & 0 & 2t^2 & at^2 \\  
at & -nbt & 0 & \dots & \dots & 0 & at^2 & (a^2-2b)t^2 
\end{array}
\hspace{-0.5mm}\right]
$}
\end{align}
and hence
\begin{align} \label{eq6.2}
A_{\bm c}(t)_1
=
\scalebox{0.82}[0.82]
{$\left[
\begin{array}{cccccccc}
1 & 0 & \dots & \dots & \dots & 0 & 0 & 0 \\
0 &&&&& -(n-2)t & -(n-1)at & -nbt \\
\vdots &&&&  \rotatebox[origin=c]{292}{\vdots} &  -(n-1)at & -nbt & 0 \\
\vdots &&& \rotatebox[origin=c]{292}{\vdots} & \hspace{3mm}  \rotatebox[origin=c]{292}{\vdots} & -nbt & 0 & \vdots \\ 
\vdots && \rotatebox[origin=c]{292}{\vdots} & \hspace{3mm}  \rotatebox[origin=c]{292}{\vdots} & \hspace{3mm}  \rotatebox[origin=c]{292}{\vdots} & \hspace{3mm}  \rotatebox[origin=c]{292}{\vdots} & \vdots & \vdots \\
0 & -(n-2)t & -(n-1)at & -nbt & \hspace{3mm}  \rotatebox[origin=c]{292}{\vdots} && 0 & 0 \\ 
0 & -(n-1)at & -nbt & 0 & \dots & 0 & 2\left(1-\frac{2}{n}\right)t^2 &  \left(1-\frac{2}{n}\right)at^2 \\  
0 & -nbt & 0 & \dots & \dots & 0 & \left(1-\frac{2}{n}\right)at^2 & \left\{\left(1-\frac{1}{n}\right)a^2-2b\right\}t^2
\end{array}
\right].
$} 
\end{align}
Here, similar to $A_{\bm c}(t)_{k}$ $(2 \leq k \leq n-2)$, we inductively define the matrix $W(t)_{k}=(w_{ij}(t)_k)_{1\leq i,j \leq n}$ $(2 \leq k \leq n-2)$ as follows; \\[3mm]
(1) $W(t)_1=(w_{ij}(t)_1)_{1 \leq i,j \leq n}$, where
\begin{align*}
w_{ij}(t)_1
=
\begin{cases}
qt, & \text{$i+j=n$, $2 \leq i, j \leq n-2$}, \\
rt, & \text{$i+j=n+1$, $2 \leq i, j \leq n-1$}, \\
st, & \text{$i+j=n+2$, $2 \leq i, j \leq n$}, \\
0, & \text{otherwise}.
\end{cases}
\end{align*}
(2)  $W(t)_{k}=(w_{ij}(t)_k)_{1\leq i,j \leq n}$ $(2 \leq k \leq n-2)={}^tS(t)_kW(t)_{k-1}S(t)_k$, where
\begin{align*}
S(t)_k=
\begin{cases}
\displaystyle \prod_{m=n-k+1}^{n}R_n\left(n-k,m;-\dfrac{w_{km}(t)_{k-1}}{qt}\right) &\hspace{-24mm}\text{($2 \leq k \leq n_0$)}, \\[3mm]
\displaystyle R_n\left(n-k,k;-\dfrac{w_{kk}(t)_{k-1}}{-2qt}\right)\prod_{m=k+1}^{n}R_n\left(n-k,m;-\dfrac{w_{km}(t)_{k-1}}{qt}\right), \hfill \text{($n_0 < k \leq n-2$)}.
\end{cases}
\end{align*}

\begin{exa}\label{exa3}

\begin{enumerate}
\item Let $n=5$. Then, 
\begin{align*}
W(t)_1 =
\scalebox{0.9}[0.9]
{$\left[
\begin{array}{ccccc}
0 & 0 & 0 & 0 & 0 \\
0 & 0 & qt & rt & st \\
0 & qt & rt & st & 0 \\
0 & rt & st & 0 & 0 \\
0 & st & 0 & 0 & 0
\end{array}
\right]
$}, \; \; 
& W(t)_{n_0}=
\scalebox{0.9}[0.9]
{$\left[
\begin{array}{ccccc}
0 & 0 & 0 & 0 & 0 \\
0 & 0 & qt & 0 & 0 \\
0 & qt & rt & (qs-r^2)t/q & -rst/q \\
0 & 0 & (qs-r^2)t/q & -(2qs-r^2)rt/q^2 & -(qs-r^2)st/q^2 \\
0 & 0 & -rst/q & -(qs-r^2)st/q^2 & rs^2t/q^2
\end{array}
\right]
$}, \\
W(t)_{n-2}&=
\scalebox{0.9}[0.9]
{$\left[
\begin{array}{ccccc}
0 & 0 & 0 & 0 & 0 \\
0 & 0 & qt & 0 & 0 \\
0 & qt & 0 & 0 & 0 \\
0 & 0 & 0 & -(2qs-r^2)rt/q^2 & -(qs-r^2)st/q^2 \\
0 & 0 & 0 & -(qs-r^2)st/q^2 & rs^2t/q^2
\end{array}
\right]
$}.
\end{align*}

\item
Let $n=6$. Then,
\begin{align*}
W(t)_1&=
\scalebox{0.9}[0.9]
{$\left[
\begin{array}{cccccc}
0 & 0 & 0 & 0 & 0 & 0 \\
0 & 0 & 0 & qt & rt & st \\
0 & 0 & qt & rt & st & 0 \\
0 & qt & rt & st & 0 & 0 \\
0 & rt & st & 0 & 0 & 0 \\
0 & st & 0 & 0 & 0 & 0
\end{array}
\right]
$}, \\
W(t)_{n_0}&=
\scalebox{0.88}[0.9]
{$\left[
\begin{array}{cccccc}
0 & 0 & 0 & 0 & 0 & 0 \\
0 & 0 & 0 & qt & 0 & 0 \\
0 & 0 & qt & 0 & 0 & 0 \\
0 & qt & 0 & (qs-r^2)t/q & -(2qs-r^2)rt/q^2 & -(qs-r^2)st/q^2 \\
0 & 0 & 0 & -(2qs-r^2)rt/q^2 & -(q^2s^2-3qr^2s+r^4)t/q^3 & (2qs-r^2)rst/q^3 \\
0 & 0 & 0 & -(qs-r^2)st/q^2 & (2qs-r^2)rst/q^3 & (qs-r^2)s^2t/q^3
\end{array}
\right]
$}, \\
W(t)_{n-2}&=
\scalebox{0.9}[0.9]
{$\left[
\begin{array}{cccccc}
0 & 0 & 0 & 0 & 0 & 0 \\
0 & 0 & 0 & qt & 0 & 0 \\
0 & 0 & qt & 0 & 0 & 0 \\
0 & qt & 0 & 0 & 0 & 0 \\
0 & 0 & 0 & 0 & -(q^2s^2-3qr^2s+r^4)t/q^3 & (2qs-r^2)rst/q^3 \\
0 & 0 & 0 & 0 & (2qs-r^2)rst/q^3 & (qs-r^2)s^2t/q^3
\end{array}
\right]
$}.
\end{align*}
\end{enumerate}
\end{exa}

In the same way, let us compute the matrix $W(t)_{n-2}$ for any integer $n\geq 7$.

\begin{lem} \label{lem6.2}
Suppose $n\geq 7$. Let $\{ x_m \}$, $\{ y_m \}$ be sequences defined by the next recurrence relations;
\begin{align*}
&x_0=0, \ x_1=-qt, \ x_2=rt, \ x_{m+2}=-\frac{r}{q}x_{m+1}-\frac{s}{q}x_m \ (m\geq 1), \\ 
&y_0=y_1=0, \ y_{m+1}=-\frac{s}{q}x_{m} \ (m \geq 1)
\end{align*}
and put
\begin{align*}
n_1=
\begin{cases}
n_0-1 & \text{$n:$ odd}, \\
n_0-2 & \text{$n:$ even}.
\end{cases}
\end{align*}
Then, for any integer $k$ such that $2 \leq k \leq n_1$, we have
\begin{align*}
w_{ij}(t)_{k}
=
\begin{cases}
0 & \text{$(i,j)=(k, \ell)$ or $(\ell, k)$ $(n-k+1\leq \ell \leq n)$}, \\
x_{\ell-n+k+2} & \text{$(i,j)=(k+1, \ell)$ or $(\ell, k+1)$ $(n-k \leq \ell \leq n-1)$}, \\
y_{\ell-n+k+2} & \text{$(i,j)=(k+2, \ell)$ or $(\ell, k+2)$ $(n-k \leq \ell \leq n-1)$}, \\
- \frac {sx_{k}} q & \text{$(i,j)=(k+1, n)$ or $(n, k+1)$}, \\
- \frac {sy_{k}} q & \text{$(i,j)=(k+2, n)$ or $(n, k+2)$}, \\
w_{ij}(t)_{k-1} & \text{otherwise}.
\end{cases}
\end{align*} 
\end{lem}

\begin{proof}
Let us prove this lemma by induction on $k$. First, suppose $k=2$. Then, by the definition of $W(t)_2$, we have
\begin{align*}
W(t)_2
=
\scalebox{0.76}[1]
{$\left[
\begin{array}{cccccccccc}
0 & \cdots & \cdots & \dots & \dots & \dots & 0 & 0 & 0 & 0 \\
\vdots &&&&&&& qt & 0 & 0 \\
\vdots &&&&&& qt & rt & \frac {(qs-r^2)t} q & - \frac {rst} q \\
\vdots &&&&& \rotatebox[origin=c]{300}{\vdots} & rt & st & - \frac {rst} q & - \frac {s^2t} q \\ 
\vdots &&&& \rotatebox[origin=c]{300}{\vdots} & \rotatebox[origin=c]{300}{\vdots} & st & 0 & 0 & 0 \\
\vdots &&& \rotatebox[origin=c]{300}{\vdots} & \rotatebox[origin=c]{300}{\vdots} & \rotatebox[origin=c]{300}{\vdots} & \rotatebox[origin=c]{300}{\vdots} & \rotatebox[origin=c]{300}{\vdots} & \rotatebox[origin=c]{300}{\vdots} & \vdots \\
0 & 0 & qt & rt & st & \rotatebox[origin=c]{300}{\vdots} & \rotatebox[origin=c]{300}{\vdots} & \rotatebox[origin=c]{300}{\vdots} && \vdots \\
0 & qt & rt & st & 0 & \rotatebox[origin=c]{300}{\vdots} & \rotatebox[origin=c]{300}{\vdots} &&& \vdots \\ 
0 & 0 & (qs-r^2) \frac t q & - \frac {rst} q & 0 & \rotatebox[origin=c]{300}{\vdots} &&&& \vdots \\  
0 & 0 & - \frac {rst} q & - \frac {s^2t} q & 0 & \cdots & \cdots & \cdots & \cdots & 0
\end{array}
\right],
$}
\end{align*}
which implies the claim of  \cref{lem6.2} for $k=2$ since
\begin{align*}
x_2=rt, \ x_3=\frac{(qs-r^2)t}{q}, \ y_2=st, \ y_3=-\frac{rst}{q}.
\end{align*}
Next, suppose  \cref{lem6.2} is true for $k=2, \cdots, m-1$ $(m-1<n_1)$. Then, since
\begin{align*}
W(t)_{m-1}
=
\scalebox{0.70}[0.7]
{$\left[
\begin{array}{ccccccccccccc}
0 & \cdots & \cdots & \cdots & \cdots & \cdots & \cdots & \cdots & \cdots & \cdots & \cdots & 0 & 0 \\
\vdots &&&&&&&&& \rotatebox[origin=c]{300}{\vdots} & \rotatebox[origin=c]{300}{\vdots} & \vdots & \vdots \\ 
\vdots &&&&&&&& qt & 0 & \cdots & 0 & 0 \\
\vdots &&&&&&& qt & x_2 & x_3 & \cdots & x_{m} & -sx_{m-1}/q \\
\vdots &&&&&& \rotatebox[origin=c]{300}{\vdots} & rt & y_2 & y_3 & \cdots & y_{m} & -sy_{m-1}/q \\
\vdots &&&&& \rotatebox[origin=c]{300}{\vdots} & \rotatebox[origin=c]{300}{\vdots} & st & 0 & 0 & \cdots & 0 & 0 \\
\vdots &&&& \rotatebox[origin=c]{300}{\vdots} & \rotatebox[origin=c]{300}{\vdots} & \rotatebox[origin=c]{300}{\vdots} & \rotatebox[origin=c]{300}{\vdots} & \rotatebox[origin=c]{300}{\vdots} &&& \rotatebox[origin=c]{300}{\vdots} & \vdots \\
\vdots &&& qt & rt & st & \rotatebox[origin=c]{300}{\vdots} & \rotatebox[origin=c]{300}{\vdots} &&& \rotatebox[origin=c]{300}{\vdots} && \vdots \\
\vdots && qt & x_2 & y_2 & 0 & \rotatebox[origin=c]{300}{\vdots} &&& \rotatebox[origin=c]{300}{\vdots} &&& \vdots \\
\vdots & \rotatebox[origin=l]{300}{\vdots} & 0 & x_3 & y_3 & 0 &&& \rotatebox[origin=c]{300}{\vdots} &&&& \vdots \\
\vdots & \rotatebox[origin=c]{300}{\vdots} & \vdots & \vdots & \vdots & \vdots && \rotatebox[origin=c]{300}{\vdots} &&&&& \vdots \\
0 & \cdots & 0 & x_{m} & y_{m} & 0 & \rotatebox[origin=c]{300}{\vdots} &&&&&& \vdots \\
0 & \cdots & 0 & -sx_{m-1}/q & -sy_{m-1}/q & 0 & \cdots & \cdots& \cdots & \cdots & \cdots & \cdots & 0
\end{array} 
\right],
$}
\end{align*}
\normalsize
we have
\begin{align*}
W(t)_{m}
=
\scalebox{0.72}[0.72]
{$\left[
\begin{array}{ccccccccccccc}
0 & \cdots & \cdots & \cdots & \cdots & \cdots & \cdots & \cdots & \cdots & \cdots & \cdots & 0 & 0 \\
\vdots &&&&&&&&& \rotatebox[origin=c]{300}{\vdots} & \rotatebox[origin=c]{300}{\vdots} & \vdots & \vdots \\ 
\vdots &&&&&&&& qt & 0 & \cdots & 0 & 0 \\
\vdots &&&&&&& qt & 0 & 0 & \cdots & 0 & 0 \\
\vdots &&&&&& \rotatebox[origin=c]{300}{\vdots} & rt & x_2^{\prime} & x_3^{\prime} & \cdots & x_{m}^{\prime} & x_{m+1}^{\prime} \\
\vdots &&&&& \rotatebox[origin=c]{300}{\vdots} & \rotatebox[origin=c]{300}{\vdots} & st & y_2^{\prime} & y_3^{\prime} & \cdots & y_{m}^{\prime} & y_{m+1}^{\prime} \\
\vdots &&&& \rotatebox[origin=c]{300}{\vdots} & \rotatebox[origin=c]{300}{\vdots} & \rotatebox[origin=c]{300}{\vdots} & \rotatebox[origin=c]{300}{\vdots} & \rotatebox[origin=c]{300}{\vdots} &&& \rotatebox[origin=c]{300}{\vdots} & \vdots \\
\vdots &&& qt & rt & st & \rotatebox[origin=c]{300}{\vdots} & \rotatebox[origin=c]{300}{\vdots} &&& \rotatebox[origin=c]{300}{\vdots} && \vdots \\
\vdots && qt & 0 & x_2^{\prime} & y_2^{\prime} & \rotatebox[origin=c]{300}{\vdots} &&& \rotatebox[origin=c]{300}{\vdots} &&& \vdots \\
\vdots & \rotatebox[origin=l]{300}{\vdots} & 0 & 0 & x_3^{\prime} & y_{3}^{\prime} &&& \rotatebox[origin=c]{300}{\vdots} &&&& \vdots \\
\vdots & \rotatebox[origin=c]{300}{\vdots} & \vdots & \vdots & \vdots & \vdots && \rotatebox[origin=c]{300}{\vdots} &&&&& \vdots \\
0 & \cdots & 0 & 0 & x_{m}^{\prime} & y_{m}^{\prime} & \rotatebox[origin=c]{300}{\vdots} &&&&&& \vdots \\
0 & \cdots & 0 & 0 & x_{m+1}^{\prime} & y_{m+1}^{\prime} & \cdots & \cdots& \cdots & \cdots & \cdots & \cdots & 0
\end{array} 
\right],
$}
\end{align*}
\normalsize
where
\begin{align*}
&x_{\ell}^{\prime}=y_{\ell}-\frac{x_{\ell}}{qt}\cdot rt=-\frac{r}{q}x_{\ell}-\frac{s}{q}x_{\ell-1}=x_{\ell+1} \ (2 \leq \ell \leq m), \\
&y_{\ell}^{\prime}=-\frac{x_{\ell}}{qt}\cdot st=-\frac{s}{q}x_{\ell}=y_{{\ell}+1} \ (2 \leq \ell \leq m), \\
&x_{m+1}^{\prime}=-\frac{s}{q}y_{m-1}+\left(-\frac{r}{q}\right)\cdot \left(-\frac{s}{q}x_{m-1}\right)=-\frac{s}{q}\left( -\frac{r}{q}x_{m-1}-\frac{s}{q}x_{m-2} \right)=-\frac{s}{q}x_{m} \\
&y_{m+1}^{\prime}=\left(-\frac{s}{q}\right)\cdot \left(-\frac{s}{q}x_{m-1}\right)=-\frac{s}{q}y_{m}.
\end{align*}
This completes the proof of  \cref{lem6.2}.
\end{proof}

\begin{prop} \label{lem6.4}
For any integer $n\geq 5$, we have
\begin{align*}
w_{ij}(t)_{n-2}
=
\begin{cases}
qt & \text{$i+j=n$ $(2 \leq i,j \leq n-2)$}, \\
x_{n-1} & \text{$(i,j)=(n-1,n-1)$}, \\
y_{n-1} & \text{$(i,j)=(n-1,n)$ or $(n,n-1)$}, \\
(-s/q)y_{n-2} & \text{$(i,j)=(n,n)$}, \\
0 & \text{otherwise}.
\end{cases}
\end{align*}
\end{prop}

\begin{proof}
 \cref{lem6.4} has been proved for $n=5$, $6$ in  \cref{exa3}. Thus, we suppose $n\geq 7$. First, by solving the recurrence relation given in  \cref{lem6.2}, we have 
\begin{align} \label{eq6.3}
x_m=
\begin{cases}
\frac{(-1)^mmr^{m-1}}{2^{m-1}q^{m-2}}t & \text{$(r^2-4qs=0)$}, \\[4mm]
\frac{(-1)^m\left\{ \left(r+\sqrt{r^2-4qs}\right)^m-\left(r-\sqrt{r^2-4qs}\right)^m \right\}}{2^mq^{m-2}\sqrt{r^2-4qs}}t & \text{$(r^2-4qs \neq 0)$}.
\end{cases}
\end{align}

Note that 
\begin{align*}
&\frac{(-1)^m\left\{ \left(r+\sqrt{r^2-4qs}\right)^m-\left(r-\sqrt{r^2-4qs}\right)^m \right\}}{2^mq^{m-2}\sqrt{r^2-4qs}}t =\frac{(-1)^mt}{2^mq^{m-2}}\cdot2\sum_{k=0}^{\lfloor \frac {m-1} 2 \rfloor}\binom{m}{2k+1}r^{m-2k-1}(r^2-4qs)^k
\end{align*}
and hence, by putting $r^2-4qs=0$, we have
\begin{align*}
\frac{(-1)^m\left\{ \left(r+\sqrt{r^2-4qs}\right)^m-\left(r-\sqrt{r^2-4qs}\right)^m \right\}}{2^mq^{m-2}\sqrt{r^2-4qs}}t=\frac{(-1)^mmr^{m-1}}{2^{m-1}q^{m-2}}t,
\end{align*}
the solution for the case of $r^2-4qs=0$. Here, suppose $n$ is odd. Then, we have $n_1= \frac {n-3} 2$ and hence by  \cref{lem6.2},
\begin{align*}
w_{ij}(t)_{n_1}
=
\begin{cases}
0 & \text{$(i,j)=((n-3)/2, \ell)$ or $(\ell, (n-3)/2)$ $((n+5)/2\leq \ell \leq n)$}, \\
x_{\ell-(n-1)/2} & \text{$(i,j)=((n-1)/2, \ell)$ or $(\ell, (n-1)/2)$ $((n+3)/2 \leq \ell \leq n-1)$}, \\
y_{\ell-(n-1)/2} & \text{$(i,j)=((n+1)/2, \ell)$ or $(\ell, (n+1)/2)$ $((n+3)/2 \leq \ell \leq n-1)$}, \\
-sx_{(n-3)/2}/q & \text{$(i,j)=((n-1)/2, n)$ or $(n, (n-1)/2)$}, \\
-sy_{(n-3)/2}/q & \text{$(i,j)=((n+1)/2, n)$ or $(n, (n+1)/2)$}, \\
w_{ij}(t)_{(n-5)/2} & \text{otherwise}.
\end{cases}
\end{align*} 
Therefore, we have the next expression of the matrix $W(t)_{n_1}$;
\[
W(t)_{n_1}
=
\scalebox{0.75}[0.75]
{$\left[
\begin{array}{ccccccccccc}
0 & \cdots & \cdots & \cdots & \cdots & \cdots & \cdots & \cdots & 0 & 0 \\
\vdots &&&&&& \rotatebox[origin=c]{300}{\vdots} && \vdots & \vdots \\ 
\vdots &&&&& qt & 0 & \cdots & 0 & 0 \\
\vdots &&&& qt & x_2 & x_3 & \cdots & x_{(n-1)/2} & -sx_{(n-3)/2}/q \\
\vdots &&& qt & rt & y_2 & y_3 & \cdots & y_{(n-1)/2} & -sy_{(n-3)/2}/q \\
\vdots && qt & x_2 & y_2 & 0 &&&& \vdots \\
\vdots & \rotatebox[origin=c]{300}{\vdots} & 0 & x_3 & y_3 &&&&& \vdots \\
\vdots && \vdots & \vdots & \vdots &&&&& \vdots \\
0 & \cdots & 0 & x_{(n-1)/2} & y_{(n-1)/2} &&&&& \vdots \\
0 & \cdots & 0 & -sx_{(n-3)/2}/q & -sy_{(n-3)/2}/q & \cdots& \cdots & \cdots & \cdots & 0
\end{array} 
\right].
$}
\]
Then, since $W(t)_{n_1+1}={}^tS(t)_{n_1+1}W(t)_{n_1}S(t)_{n_1+1}$, where
\begin{align*}
S(t)_{n_1+1}=\prod_{m=(n+3)/2}^{n}R_n\left(\frac{n+1}{2},m;-\dfrac{w_{(n-1)/2, m}(t)_{n_1}}{qt}\right) 
\end{align*}
and
$w_{n-1, (n+1)/2}(t)_{n_1+1}=w_{(n+1)/2, n-1}(t)_{n_1+1}=x_{(n+1)/2}$,
we have
\begin{small}
\begin{align*}
w_{n-1, n-1}(t)_{n_1+1} & =-\frac{x_{(n-1)/2}}{qt}\cdot y_{(n-1)/2}-\frac{x_{(n-1)/2}}{qt}\cdot x_{(n+1)/2} \\[1mm]
&=-\frac{x_{(n-1)/2}}{qt}\cdot \left(-\frac{s}{q}x_{(n-3)/2} \right)-\frac{x_{(n-1)/2}}{qt}\cdot x_{(n+1)/2} \\[1mm]
&=-\frac{1}{qt} \cdot \frac{(-1)^{(n-1)/2}\left\{ \left( r+\sqrt{r^2-4qs} \right)^{(n-1)/2}-\left( r-\sqrt{r^2-4qs} \right)^{(n-1)/2} \right\}}{2^{(n-1)/2}q^{(n-5)/2}\sqrt{r^2-4qs}}t \\[2mm]
& \hspace{7mm} \times \left( -\frac{s}{q} \right)\cdot \frac{(-1)^{(n-3)/2}\left\{ \left( r+\sqrt{r^2-4qs} \right)^{(n-3)/2}-\left( r-\sqrt{r^2-4qs} \right)^{(n-3)/2} \right\}}{2^{(n-3)/2}q^{(n-7)/2}\sqrt{r^2-4qs}}t \\[2mm]
& \hspace{3mm} -\frac{1}{qt} \cdot \frac{(-1)^{(n-1)/2}\left\{ \left( r+\sqrt{r^2-4qs} \right)^{(n-1)/2}-\left( r-\sqrt{r^2-4qs} \right)^{(n-1)/2} \right\}}{2^{(n-1)/2}q^{(n-5)/2}\sqrt{r^2-4qs}}t \\[2mm]
& \hspace{7mm} \times \frac{(-1)^{(n+1)/2}\left\{ \left( r+\sqrt{r^2-4qs} \right)^{(n+1)/2}-\left( r-\sqrt{r^2-4qs} \right)^{(n+1)/2} \right\}}{2^{(n+1)/2}q^{(n-3)/2}\sqrt{r^2-4qs}}t \\[2mm]
&=-s\biggl\{\left( r+\sqrt{r^2-4qs} \right)^{n-2}-\left(4qs\right)^{(n-3)/2}\left(r+\sqrt{r^2-4qs}\right) \\[1mm]
&\hspace{7mm}-\left(4qs\right)^{(n-3)/2}\left(r-\sqrt{r^2-4qs}\right)+\left( r-\sqrt{r^2-4qs} \right)^{n-2}\biggr\}t \biggl/\left\{2^{n-2}q^{n-4}(r^2-4qs)\right\} \\[1mm]
&\hspace{3.5mm}+\biggl\{\left( r+\sqrt{r^2-4qs} \right)^{n}-\left(4qs\right)^{(n-1)/2}\left(r+\sqrt{r^2-4qs}\right) \\[1mm]
&\hspace{7mm}-\left(4qs\right)^{(n-1)/2}\left(r-\sqrt{r^2-4qs}\right)+\left( r-\sqrt{r^2-4qs} \right)^{n}\biggr\}t \biggl/\left\{2^{n}q^{n-3}(r^2-4qs)\right\} \\[1mm]
\end{align*}
\begin{align*}
&=\Biggl\{\left( r+\sqrt{r^2-4qs} \right)^{n-2}\left\{\left( r+\sqrt{r^2-4qs} \right)^{2}-4qs\right\} \\[1mm]
&\hspace{7mm}+\left( r-\sqrt{r^2-4qs} \right)^{n-2}\left\{\left( r-\sqrt{r^2-4qs} \right)^{2}-4qs\right\}\Biggr\}t\biggl/\left\{2^{n}q^{n-3}(r^2-4qs)\right\} \\[2mm]
&=\frac{2\sqrt{r^2-4qs}\left\{ \left( r+\sqrt{r^2-4qs} \right)^{n-1}-\left( r-\sqrt{r^2-4qs} \right)^{n-1} \right\}}{2^{n}q^{n-3}(r^2-4qs)}t \\[2mm]
&=\frac{\left\{ \left( r+\sqrt{r^2-4qs} \right)^{n-1}-\left( r-\sqrt{r^2-4qs} \right)^{n-1} \right\}}{2^{n-1}q^{n-3}\sqrt{r^2-4qs}}t =x_{n-1}.
\end{align*}
\end{small}
Similarly, we have
\begin{align*}
w_{n-1, n}(t)_{n_1+1} &=w_{n, n-1}(t)_{n_1+1}=-\frac{x_{(n-1)/2}}{qt}\cdot\frac{-sy_{(n-3)/2}}{q}-\frac{-sx_{(n-3)/2}}{q^2t} x_{(n+1)/2} \\[1mm]
&=-\frac{s^2}{q^3t}\cdot\frac{(-1)^{(n-1)/2}\left\{ \left(r+\sqrt{r^2-4qs}\right)^{(n-1)/2}-\left(r-\sqrt{r^2-4qs}\right)^{(n-1)/2} \right\}}{2^{(n-1)/2}q^{(n-5)/2}\sqrt{r^2-4qs}}t \\[1mm]
&\hspace{10mm} \times \frac{(-1)^{(n-5)/2}\left\{ \left(r+\sqrt{r^2-4qs}\right)^{(n-5)/2}-\left(r-\sqrt{r^2-4qs}\right)^{(n-5)/2}  \right\}}{2^{(n-5)/2}q^{(n-9)/2}\sqrt{r^2-4qs}}t \\[1mm]
&\hspace{3mm} +\frac{s}{q^2t}\cdot \frac{(-1)^{(n-3)/2}\left\{ \left(r+\sqrt{r^2-4qs}\right)^{(n-3)/2}-\left(r-\sqrt{r^2-4qs}\right)^{(n-3)/2} \right\}}{2^{(n-3)/2}q^{(n-7)/2}\sqrt{r^2-4qs}}t \\[1mm]
&\hspace{10mm} \times \frac{(-1)^{(n+1)/2}\left\{ \left(r+\sqrt{r^2-4qs}\right)^{(n+1)/2}-\left(r-\sqrt{r^2-4qs}\right)^{(n+1)/2}  \right\}}{2^{(n+1)/2}q^{(n-3)/2}\sqrt{r^2-4qs}}t \\[1mm]
&=-\frac{s^2}{q}\cdot\biggl\{(r+\sqrt{r^2-4qs})^{n-3}-(4qs)^{(n-5)/2}(r+\sqrt{r^2-4qs})^2 \\
&\hspace{7mm}-(4qs)^{(n-5)/2}(r-\sqrt{r^2-4qs})^2+(r-\sqrt{r^2-4qs})^{n-3}\biggr\}t\biggl/\left\{2^{n-3}q^{n-5}(r^2-4qs)\right\} \\
&\hspace{3mm}+\frac{s}{q}\cdot\biggl\{(r+\sqrt{r^2-4qs})^{n-1}-(4qs)^{(n-3)/2}(r+\sqrt{r^2-4qs})^2 \\
&\hspace{7mm}-(4qs)^{(n-3)/2}(r-\sqrt{r^2-4qs})^2+(r-\sqrt{r^2-4qs})^{n-1}\biggr\}t\biggl/\left\{2^{n-1}q^{n-4}(r^2-4qs)\right\} \\
&=\frac{s}{q}\cdot\frac{(r+\sqrt{r^2-4qs})^{n-2}-(r-\sqrt{r^2-4qs})^{n-2}}{2^{n-2}q^{n-4}\sqrt{r^2-4qs}}t \\
&=-\frac{s}{q}x_{n-2}=y_{n-1}.
\end{align*}
Moreover, since $w_{n, (n+1)/2}(t)_{n_1+1}=w_{(n+1)/2,n}(t)_{n_1+1}=-sx_{(n-1)/2}/q$, we have
\begin{align*}
w_{n,n}(t)_{n_1+1} &=-\frac{-sx_{(n-3)/2}}{q^2t}\cdot\frac{-sy_{(n-3)/2}}{q}-\frac{-sx_{(n-3)/2}}{q^2t}\cdot\frac{-sx_{(n-1)/2}}{q} \\[1mm]
&=\frac{s^3}{q^4t}\cdot\frac{(-1)^{(n-3)/2}\left\{ \left(r+\sqrt{r^2-4qs}\right)^{(n-3)/2}-\left(r-\sqrt{r^2-4qs}\right)^{(n-3)/2} \right\}}{2^{(n-3)/2}q^{(n-7)/2}\sqrt{r^2-4qs}}t \\[1mm]
&\hspace{21mm} \times \frac{(-1)^{(n-5)/2}\left\{ \left(r+\sqrt{r^2-4qs}\right)^{(n-5)/2}-\left(r-\sqrt{r^2-4qs}\right)^{(n-5)/2}  \right\}}{2^{(n-5)/2}q^{(n-9)/2}\sqrt{r^2-4qs}}t \\[1mm]
&\hspace{3mm} -\frac{s^2}{q^3t}\cdot \frac{(-1)^{(n-3)/2}\left\{ \left(r+\sqrt{r^2-4qs}\right)^{(n-3)/2}-\left(r-\sqrt{r^2-4qs}\right)^{(n-3)/2} \right\}}{2^{(n-3)/2}q^{(n-7)/2}\sqrt{r^2-4qs}}t \\[1mm]
&\hspace{21mm} \times \frac{(-1)^{(n-1)/2}\left\{ \left(r+\sqrt{r^2-4qs}\right)^{(n-1)/2}-\left(r-\sqrt{r^2-4qs}\right)^{(n-1)/2}  \right\}}{2^{(n-1)/2}q^{(n-5)/2}\sqrt{r^2-4qs}}t \\
& =-\frac{s^3}{q^2}\cdot\biggl\{(r+\sqrt{r^2-4sq})^{n-4}-(4qs)^{(n-5)/2}(r+\sqrt{r^2-4qs}) \\
&\hspace{7mm}-(4qs)^{(n-5)/2}(r-\sqrt{r^2-4qs})+(r-\sqrt{r^2-4qs})^{n-4}\biggr\}t\biggl/\left\{2^{n-4}q^{n-6}(r^2-4qs)\right\} \\
&\hspace{3mm}+\frac{s^2}{q^2}\cdot\biggl\{(r+\sqrt{r^2-4sq})^{n-2}-(4qs)^{(n-3)/2}(r+\sqrt{r^2-4qs}) \\
&\hspace{7mm}-(4qs)^{(n-3)/2}(r-\sqrt{r^2-4qs})+(r-\sqrt{r^2-4qs})^{n-2}\biggr\}t\biggl/\left\{2^{n-2}q^{n-5}(r^2-4qs)\right\} \\
&=\frac{s^2}{q^2}\cdot\frac{(r+\sqrt{r^2-4sq})^{n-3}-(r-\sqrt{r^2-4qs})^{n-3}}{2^{n-3}q^{n-5}\sqrt{r^2-4qs}}t \\[2mm]
&=-\frac{s}{q}\left(-\frac{s}{q}x_{n-3}\right)=-\frac{s}{q}y_{n-2}.
\end{align*}
\normalsize
Therefore, we can express the matrix $W(t)_{n_0}$ $(=W(t)_{n_1+1})$ as follows;
\begin{align*}
W(t)_{n_0}=
\scalebox{1}[1]
{$\left[
\begin{array}{ccccccccccc}
0 & \cdots & \cdots & \cdots & \cdots & \cdots & \cdots & \cdots & 0 & 0 \\
\vdots &&&&&& \rotatebox[origin=c]{300}{\vdots} && \vdots & \vdots \\ 
\vdots &&&&& qt & 0 & \cdots & 0 & 0 \\
\vdots &&&& qt & 0 & 0 & \cdots & 0 & 0 \\
\vdots &&& qt & * & * & * & \cdots & * & * \\
\vdots && qt & 0 & * & * & * & \cdots & * & * \\
\vdots & \rotatebox[origin=c]{300}{\vdots} & 0 & 0 & * & * & * & \cdots & * & * \\
\vdots && \vdots & \vdots & \vdots & \vdots & \vdots & \ddots & \vdots & \vdots \\
0 & \cdots & 0 & 0 & * & * & * & \cdots & x_{n-1} & y_{n-1} \\
0 & \cdots & 0 & 0 & * & * & * & \cdots & y_{n-1} & (-s/q)y_{n-2}
\end{array} 
\right].
$}
\end{align*}
Then, by the definition of the matrix $W_{k}$ $(n_0< k \leq n-2)$, we have
\begin{align*}
W(t)_{n-2}=
\scalebox{1}[1]
{$\left[
\begin{array}{cccccccccccccc}
0 & \cdots & \cdots & \cdots & \cdots & \cdots & \cdots & 0 & 0 \\
\vdots &&&&& \rotatebox[origin=c]{310}{\vdots} && \vdots & \vdots \\ 
\vdots &&&& qt & 0 & \cdots & 0 & 0 \\
\vdots &&& qt & 0 & 0 & \cdots & 0 & 0 \\
\vdots && qt & 0 & 0 & 0 & \cdots & 0 & 0 \\
\vdots & \rotatebox[origin=c]{310}{\vdots} & 0 & 0 & 0 & 0 & \cdots & 0 & 0 \\
\vdots && \vdots & \vdots & \vdots & \vdots & \ddots & \vdots & \vdots \\
0 && 0 & 0 & 0 & 0 & \cdots & x_{n-1} & y_{n-1} \\[1.5mm]
0 & \cdots & 0 & 0 & 0 & 0 & \cdots & y_{n-1} & (-s/q)y_{n-2} \\
\end{array} 
\right],
$}
\end{align*}
which completes the proof of  \cref{lem6.4} for odd $n$.

Next, suppose $n$ is even. Then, we have $n_1=(n-4)/2$ and hence by  \cref{lem6.2}, we have
\begin{align*}
w_{ij}(t)_{n_1}
=
\begin{cases}
0 & \text{$(i,j)=((n-4)/2, \ell)$ or $(\ell, (n-4)/2)$ $((n+6)/2\leq \ell \leq n)$}, \\
x_{\ell-n/2} & \text{$(i,j)=((n-2)/2, \ell)$ or $(\ell, (n-2)/2)$ $((n+4)/2 \leq \ell \leq n-1)$}, \\
y_{\ell-n/2} & \text{$(i,j)=(n/2, \ell)$ or $(\ell, n/2)$ $((n+4)/2 \leq \ell \leq n-1)$}, \\
-sx_{(n-4)/2}/q & \text{$(i,j)=((n-2)/2, n)$ or $(n, (n-2)/2)$}, \\
-sy_{(n-4)/2}/q & \text{$(i,j)=(n/2, n)$ or $(n, n/2)$}, \\
w_{ij}(t)_{(n-6)/2} & \text{otherwise}.
\end{cases} 
\end{align*} 
\normalsize
Thus, we have the next expression of the matrix $W(t)_{n_1}$;
\begin{align*}
W(t)_{n_1}
=
\scalebox{0.8}[1]
{$\left[
\begin{array}{cccccccccccc}
0 & \cdots & \cdots & \cdots & \cdots & \cdots & \cdots & \cdots & \cdots & 0 & 0 \\
\vdots &&&&&&& \rotatebox[origin=c]{300}{\vdots} && \vdots & \vdots \\ 
\vdots &&&&&& qt & 0 & \cdots & 0 & 0 \\
\vdots &&&&& qt & x_2 & x_3 & \cdots & x_{(n-2)/2} & -sx_{(n-4)/2}/q \\
\vdots &&&& qt & rt & y_2 & y_3 & \cdots & y_{(n-2)/2} & -sy_{(n-4)/2}/q \\
\vdots &&& qt & rt & st & 0 & 0 & \cdots & 0 & 0 \\
\vdots && qt & x_2 & y_2 & 0 & 0 &&&& \vdots \\
\vdots & \rotatebox[origin=c]{300}{\vdots} & 0 & x_3 & y_3 & 0 &&&&& \vdots \\
\vdots && \vdots & \vdots & \vdots & \vdots &&&&& \vdots \\
0 & \cdots & 0 & x_{(n-2)/2} & y_{(n-2)/2} & 0 &&&&& \vdots \\
0 & \cdots & 0 & -sx_{(n-4)/2}/q & -sy_{(n-4)/2}/q & 0 & \cdots& \cdots & \cdots & \cdots & 0
\end{array} 
\right].
$}
\end{align*}

Then, for any integer $\ell$ $((n+4)/2 \leq \ell \leq n)$, let us put 
\begin{align*}
V(t)_{\ell}=(v_{ij}(t)_{\ell})_{1\leq i,j \leq n}={}^tR(t)_{\ell}W(t)_{n_1}R(t)_{\ell},
\end{align*}
where
\begin{align*}
R(t)_{\ell}=\prod_{m=(n+4)/2}^{\ell}R_n\left((n+2)/2,m;-\dfrac{w_{(n-2)/2, m}(t)_{n_1}}{qt}\right).
\end{align*}
Note that we have $V(t)_n=W(t)_{n_1+1}$.

\begin{clm}
For any integer $\ell^{\prime}$ $((n+4)/2 \leq \ell^{\prime} \leq n-1)$, we have  
\begin{align*}
v_{ij}(t)_{\ell^{\prime}}
=
\begin{cases}
0 & \text{$(i,j)=((n-2)/2,\ell)$ or $(\ell,(n-2)/2)$ $((n+4)/2 \leq \ell \leq \ell^{\prime})$}, \\
x_{\ell-(n-2)/2} & \text{$(i,j)=(n/2, \ell)$ or $(\ell, n/2)$ $((n+2)/2 \leq \ell \leq \ell^{\prime})$}, \\
y_{\ell-(n-2)/2} & \text{$(i,j)=((n+2)/2, \ell)$ or $(\ell, (n+2)/2)$ $((n+2)/2 \leq \ell \leq \ell^{\prime})$}, \\
\left( -x_{\ell-n/2}/qt \right)^2y_2 & \text{$(i,j)=(\ell, \ell)$ $((n+4)/2 \leq \ell \leq \ell^{\prime})$}, \\
(-x_{m-n/2}/qt) \\
\hspace{6.5mm}\times y_{\ell-(n-2)/2} & \text{$(i,j)=(\ell,m)$ or $(m, \ell)$ $((n+4)/2 \leq \ell<m \leq \ell^{\prime})$}, \\
w_{ij}(t)_{(n-4)/2} & \text{otherwise}.
\end{cases}
\end{align*}
\normalsize
\end{clm}

\begin{proof}
To ease notation, let us put $k=(n-2)/2$. Then, by definition, 
\normalsize
\begin{align*}
V(t)_{(n+4)/2}= 
\scalebox{0.75}[0.95]
{$\left[
\begin{array}{cccccccccccc}
0 & \cdots & \cdots & \cdots & \cdots & \cdots & \cdots & \cdots & \cdots & 0 & 0 \\
\vdots &&&&&&& \rotatebox[origin=c]{300}{\vdots} && \vdots & \vdots \\ 
\vdots &&&&&& qt & 0 & \cdots & 0 & 0 \\
\vdots &&&&& qt & 0 & x_3 & \cdots & x_{k} & -sx_{k-1}/q \\
\vdots &&&& qt & x_2 & x_3 & y_3 & \cdots & y_{k} & -sy_{k-1}/q \\
\vdots &&& qt & x_2 & y_2 & y_3 & 0 & \cdots & 0 & 0 \\
\vdots && qt & 0 & x_3 & y_3 & \left( -x_2/qt \right)^2y_2 &&&& \vdots \\
\vdots & \rotatebox[origin=c]{300}{\vdots} & 0 & x_3 & y_3 & 0 &&&&& \vdots \\
\vdots & & \vdots & \vdots & \vdots & \vdots &&&&& \vdots \\
0 & \cdots & 0 & x_{k} & y_{k} & 0 &&&&& \vdots \\
0 & \cdots & 0 & -sx_{k-1}/q & -sy_{k-1}/q & 0 & \cdots& \cdots & \cdots & \cdots & 0
\end{array} 
\right]
$}
\end{align*}
\normalsize
and we get Claim for $V(t)_{(n+4)/2}$. Here, suppose Claim is true for $V(t)_{\ell^{\prime}-1}$ and hence
\begin{align*}
V(t)_{\ell^{\prime}-1}= 
\scalebox{0.63}[0.8]
{$\left[
\begin{array}{cccccccccccccc}
0 & \cdots & \cdots & \cdots & \cdots & \cdots & \cdots & \cdots & 0 & 0 & \cdots \\
\vdots &&&&& \rotatebox[origin=c]{300}{\vdots} &&& \vdots & \vdots & \\ 
\vdots &&&& qt & 0 & 0 & \cdots & 0 & x_{\ell^{\prime}-k-1} & \cdots \\
\vdots &&& qt & x_2 & x_3 & x_4 & \cdots & x_{\ell^{\prime}-k-1} & y_{\ell^{\prime}-k-1} & \cdots \\
\vdots && qt & x_2 & y_2 & y_3 & y_4 & \cdots & y_{\ell^{\prime}-k-1} & 0 & \cdots \\
\vdots & \rotatebox[origin=c]{300}{\vdots} & 0 & x_3 & y_3 & \left( -x_2/qt \right)^2y_2 & \left( -x_3/qt \right)y_3 & \cdots & (-x_{\ell^{\prime}-k-2}/qt)y_3 & 0 & \cdots \\
\vdots && 0 & x_4 & y_4 & \left( -x_3/qt \right)y_3 & \left( -x_3/qt \right)^2y_2 & \cdots & (-x_{\ell^{\prime}-k-2}/qt)y_4 & 0 & \cdots \\
\vdots && \vdots & \vdots & \vdots & \vdots & \vdots & \ddots & \vdots & \vdots & \\
0 && 0 & x_{\ell^{\prime}-k-1} & y_{\ell^{\prime}-k-1} & (-x_{\ell^{\prime}-k-2}/qt)y_3 & (-x_{\ell^{\prime}-k-2}/qt)y_4 & \cdots & (-x_{\ell^{\prime}-k-2}/qt)^2y_2 & 0 & \cdots \\
0 & \cdots & x_{\ell^{\prime}-k-1} & y_{\ell^{\prime}-k-1} & 0 & 0 & 0 & \cdots & 0 & 0 & \cdots \\
\vdots && \vdots & \vdots & \vdots & \vdots & \vdots && \vdots & \vdots  
\end{array} 
\right].
$}
\end{align*}
Thus, by a direct computation, we have
\begin{align*}
V(t)_{\ell^{\prime}}=
\scalebox{0.7}[0.9]
{$\left[
\begin{array}{cccccccccccccc}
0 & \cdots & \cdots & \cdots & \cdots & \cdots & \cdots & \cdots & 0 & 0 & \cdots \\
\vdots &&&&& \rotatebox[origin=c]{300}{\vdots} &&& \vdots & \vdots & \\ 
\vdots &&&& qt & 0 & 0 & \cdots & 0 & x_{\ell^{\prime}-k} & \cdots \\
\vdots &&& qt & x_2 & x_3 & x_4 & \cdots & x_{\ell^{\prime}-k} & y_{\ell^{\prime}-k} & \cdots \\
\vdots && qt & x_2 & y_2 & y_3 & y_4 & \cdots & y_{\ell^{\prime}-k} & 0 & \cdots \\
\vdots & \rotatebox[origin=c]{300}{\vdots} & 0 & x_3 & y_3 & \left( -x_2/qt \right)^2y_2 & \left( -x_3/qt \right)y_3 & \cdots & (-x_{\ell^{\prime}-k-1}/qt)y_3 & 0 & \cdots \\
\vdots && 0 & x_4 & y_4 & \left( -x_3/qt \right)y_3 & \left( -x_3/qt \right)^2y_2 & \cdots & (-x_{\ell^{\prime}-k-1}/qt)y_4 & 0 & \cdots \\
\vdots && \vdots & \vdots & \vdots & \vdots & \vdots & \ddots & \vdots & \vdots & \\
0 && 0 & x_{\ell^{\prime}-k} & y_{\ell^{\prime}-k} & (-x_{\ell^{\prime}-k-1}/qt)y_3 & (-x_{\ell^{\prime}-k-1}/qt)y_4 & \cdots & (-x_{\ell^{\prime}-k-1}/qt)^2y_2 & 0 & \cdots \\
0 & \cdots & x_{\ell^{\prime}-k} & y_{\ell^{\prime}-k} & 0 & 0 & 0 & \cdots & 0 & 0 & \cdots \\
\vdots && \vdots & \vdots & \vdots & \vdots & \vdots && \vdots & \vdots  
\end{array} 
\right]
$}
\end{align*}
\normalsize
and we get Claim by induction on $\ell^{\prime}$.
\end{proof}

By above Claim, 
\begin{align*}
V(t)_{n-1}= 
\scalebox{0.7}[0.7]
{$\left[
\begin{array}{cccccccccccccc}
0 & \cdots & \cdots & \cdots & \cdots & \cdots & \cdots & \cdots & 0 & 0 \\
\vdots &&&&& \rotatebox[origin=c]{300}{\vdots} &&& \vdots & \vdots \\ 
\vdots &&&& qt & 0 & 0 & \cdots & 0 & -sx_{k-1}/q \\
\vdots &&& qt & x_2 & x_3 & x_4 & \cdots & x_{k+1} & -sy_{k-1}/q \\
\vdots && qt & x_2 & y_2 & y_3 & y_4 & \cdots & y_{k+1} & 0 \\
\vdots & \rotatebox[origin=c]{300}{\vdots} & 0 & x_3 & y_3 & \left( -x_2/qt \right)^2y_2 & \left( -x_3/qt \right)y_3 & \cdots & (-x_{k}/qt)y_3 & 0 \\
\vdots && 0 & x_4 & y_4 & \left( -x_3/qt \right)y_3 & \left( -x_3/qt \right)^2y_2 & \cdots & (-x_{k}/qt)y_4 & 0 \\
\vdots && \vdots & \vdots & \vdots & \vdots & \vdots & \ddots & \vdots & \vdots \\
0 && 0 & x_{k+1} & y_{k+1} & (-x_{k}/qt)y_3 & (-x_{k}/qt)y_4 & \cdots & (-x_{k}/qt)^2y_2 & 0 \\
0 & \cdots & -sx_{k-1}/q & -sy_{k-1}/q & 0 & 0 & 0 & \cdots & 0 & 0 \\
\end{array} 
\right]
$}
\end{align*}
\normalsize
and hence, by the definition of $V(t)_n$, we have
\begin{align*}
&W(t)_{n_1+1}=V(t)_{n}= \\
&\scalebox{0.7}[0.7]
{$\left[
\begin{array}{cccccccccccccc}
0 & \cdots & \cdots & \cdots & \cdots & \cdots & \cdots & \cdots & 0 & 0 \\
\vdots &&&&& \rotatebox[origin=c]{300}{\vdots} &&& \vdots & \vdots \\ 
\vdots &&&& qt & 0 & 0 & \cdots & 0 & 0 \\
\vdots &&& qt & x_2 & x_3 & x_4 & \cdots & x_{k+1} & -sx_{k}/q \\
\vdots && qt & x_2 & y_2 & y_3 & y_4 & \cdots & y_{k+1} & -sy_{k}/q \\
\vdots & \rotatebox[origin=c]{300}{\vdots} & 0 & x_3 & y_3 & \left( -x_2/qt \right)^2y_2 & \left( -x_3/qt \right)y_3 & \cdots & (-x_{k}/qt)y_3 & (x_{k-1}/q^2t^2)y_2y_3 \\
\vdots && 0 & x_4 & y_4 & \left( -x_3/qt \right)y_3 & \left( -x_3/qt \right)^2y_2 & \cdots & (-x_{k}/qt)y_4 & (x_{k-1}/q^2t^2)y_2y_4 \\
\vdots && \vdots & \vdots & \vdots & \vdots & \vdots & \ddots & \vdots & \vdots \\
0 && 0 & x_{k+1} & y_{k+1} & (-x_{k}/qt)y_3 & (-x_{k}/qt)y_4 & \cdots & (-x_{k}/qt)^2y_2 & (x_{k-1}/q^2t^2)y_2y_{k+1} \\[1.5mm]
0 & \cdots & 0 & -sx_{k}/q & -sy_{k}/q & (x_{k-1}/q^2t^2)y_2y_3 & (x_{k-1}/q^2t^2)y_2y_4 & \cdots & (x_{k-1}/q^2t^2)y_2y_{k+1} & (x_{k-1}/q^2t^2)^2y_2^3 \\
\end{array} 
\right].
$}
\end{align*}
Therefore, by the definition of the matrix $W(t)_{n_1+2}=W(t)_{n_0}$, we have
\begin{small}
\begin{align*}
w_{n-1,n-1}(t)_{n_0} &=\left(-\frac{x_{(n-2)/2}}{qt}\right)^2y_2-\frac{\left(x_{n/2}\right)^2}{qt} \\[2mm]
&=\frac{s\left\{\left(r+\sqrt{r^2-4qs}\right)^{n-2}+\left(r-\sqrt{r^2-4qs}\right)^{n-2}-2(4qs)^{(n-2)/2}\right\}}{2^{n-2}q^{n-4}\left(r^2-4qs\right)}t \\
&\hspace{45mm}-\frac{\left(r+\sqrt{r^2-4qs}\right)^{n}+\left(r-\sqrt{r^2-4qs}\right)^{n}-2(4qs)^{n/2}}{2^nq^{n-3}\left(r^2-4qs\right)}t \\
&=\frac{\left(r+\sqrt{r^2-4qs}\right)^{n-2}\left\{ 4qs-\left(r+\sqrt{r^2-4qs}\right)^2\right\}}{2^nq^{n-3}\left(r^2-4qs\right)}t \\
&\hspace{45mm}+\frac{\left(r-\sqrt{r^2-4qs}\right)^{n-2}\left\{ 4qs-\left(r-\sqrt{r^2-4qs}\right)^2 \right\}}{2^nq^{n-3}\left(r^2-4qs\right)}t \\
&=\frac{\left(r+\sqrt{r^2-4qs}\right)^{n-2}\left(-2\sqrt{r^2-4qs}\right)\left(r+\sqrt{r^2-4qs}\right)}{2^nq^{n-3}\left(r^2-4qs\right)}t \\
&\hspace{40mm}+\frac{\left(r-\sqrt{r^2-4qs}\right)^{n-2}\left(2\sqrt{r^2-4qs}\right)\left(r-\sqrt{r^2-4qs}\right)}{2^nq^{n-3}\left(r^2-4qs\right)}t =x_{n-1}. 
\end{align*}
\end{small}
Similarly, we have
\begin{align*}
&w_{n-1,n}(t)_{n_0}=w_{n,n-1}(t)_{n_0}=\frac{x_{(n-4)/2}}{q^2t^2}y_2y_{n/2}+\frac{sx_{(n-2)/2}x_{n/2}}{q^2t} \\
&=\frac{x_{(n-4)/2}}{q^2t^2}y_2\left(-\frac{s}{q}x_{(n-2)/2}\right)+\frac{sx_{(n-2)/2}x_{n/2}}{q^2t} \\
&=-\frac{s^2}{q^3}\cdot\frac{-\left\{\left(r+\sqrt{r^2-4qs}\right)^{n-3}+\left(r-\sqrt{r^2-4qs}\right)^{n-3}-2r\left( 4qs \right)^{(n-4)/2}\right\}}{2^{n-3}q^{n-7}\left(r^2-4qs\right)}t \\
&\hspace{17mm}+\frac{s}{q^2}\cdot\frac{-\left\{\left(r+\sqrt{r^2-4qs}\right)^{n-1}+\left(r-\sqrt{r^2-4qs}\right)^{n-1}-2r\left( 4qs \right)^{(n-2)/2}\right\}}{2^{n-1}q^{n-5}\left(r^2-4qs\right)}t \\
&=\frac{s\left\{4qs\left(r+\sqrt{r^2-4qs}\right)^{n-3}+4qs\left(r-\sqrt{r^2-4qs}\right)^{n-3}-2r\left( 4qs \right)^{(n-2)/2}\right\}}{2^{n-1}q^{n-3}\left(r^2-4qs\right)}t \\
&\hspace{22mm}-\frac{s\left\{\left(r+\sqrt{r^2-4qs}\right)^{n-1}+\left(r-\sqrt{r^2-4qs}\right)^{n-1}-2r\left( 4qs \right)^{(n-2)/2}\right\}}{2^{n-1}q^{n-3}\left(r^2-4qs\right)}t \\
&=\frac{s\left(r+\sqrt{r^2-4qs}\right)^{n-3}\left(-2\sqrt{r^2-4qs}\right)\left(r+\sqrt{r^2-4qs}\right)}{2^{n-1}q^{n-3}\left(r^2-4qs\right)}t \\
&\hspace{25mm}+\frac{s\left(r-\sqrt{r^2-4qs}\right)^{n-3}\left(2\sqrt{r^2-4qs}\right)\left(r-\sqrt{r^2-4qs}\right)}{2^{n-1}q^{n-3}\left(r^2-4qs\right)}t \\
&=-\frac{s}{q}\cdot\frac{\left(r+\sqrt{r^2-4qs}\right)^{n-2}-\left(r-\sqrt{r^2-4qs}\right)^{n-2}}{2^{n-2}q^{n-4}\sqrt{r^2-4qs}}t=y_{n-1}, 
\end{align*}

\begin{small}
\begin{align*}
w_{n,n}(t)_{n_0} &=\left(\frac{x_{(n-4)/2}}{q^2t^2}\right)^2y_2^3-\frac{s^2\left(x_{(n-2)/2}\right)^2}{q^3t} \\
&=\frac{s^3\left\{\left(r+\sqrt{r^2-4qs}\right)^{n-4}+\left(r-\sqrt{r^2-4qs}\right)^{n-4}-2(4qs)^{(n-4)/2}\right\}}{2^{n-4}q^{n-4}(r^2-4qs)}t \\
&\hspace{30mm}-\frac{s^2\left\{\left(r+\sqrt{r^2-4qs}\right)^{n-2}+\left(r-\sqrt{r^2-4qs}\right)^{n-2}-2(4qs)^{(n-2)/2}\right\}}{2^{n-2}q^{n-3}(r^2-4qs)}t \\
&=\frac{s^2\left\{4qs\left(r+\sqrt{r^2-4qs}\right)^{n-4}+4qs\left(r-\sqrt{r^2-4qs}\right)^{n-4}-2(4qs)^{(n-2)/2}\right\}}{2^{n-2}q^{n-3}(r^2-4qs)}t \\
&\hspace{30mm}-\frac{s^2\left\{\left(r+\sqrt{r^2-4qs}\right)^{n-2}+\left(r-\sqrt{r^2-4qs}\right)^{n-2}-2(4qs)^{(n-2)/2}\right\}}{2^{n-2}q^{n-3}(r^2-4qs)}t \\
&=\frac{s^2\left(r+\sqrt{r^2-4qs}\right)^{n-4}\left(-2\sqrt{r^2-4qs}\right)\left(r+\sqrt{r^2-4qs}\right)}{2^{n-2}q^{n-3}\left(r^2-4qs\right)}t \\
&\hspace{25mm}+\frac{s^2\left(r-\sqrt{r^2-4qs}\right)^{n-4}\left(2\sqrt{r^2-4qs}\right)\left(r-\sqrt{r^2-4qs}\right)}{2^{n-2}q^{n-3}\left(r^2-4qs\right)}t \\
&=\left(-\frac{s}{q} \right)\cdot\left(-\frac{s}{q} \right) \cdot \frac{-\left\{\left(r+\sqrt{r^2-4qs}\right)^{n-3}-\left(r-\sqrt{r^2-4qs}\right)^{n-3}\right\}}{2^{n-3}q^{n-5}\sqrt{r^2-4qs}}t =-\frac{s}{q}y_{n-2}.
\end{align*}

\end{small}

Therefore, we can express the matrix $W(t)_{n_0}$ as follows;
\begin{align*}
W(t)_{n_0}=
\scalebox{0.75}[0.75]
{$\left[
\begin{array}{cccccccccccccc}
0 & \cdots & \cdots & \cdots & \cdots & \cdots & \cdots & 0 & 0 \\
\vdots &&&&& \rotatebox[origin=c]{305}{\vdots} && \vdots & \vdots \\ 
\vdots &&&& qt & 0 & \cdots & 0 & 0 \\
\vdots &&& qt & 0 & 0 & \cdots & 0 & 0 \\
\vdots && qt & 0 & * & * & \cdots & * & * \\
\vdots & \rotatebox[origin=c]{305}{\vdots} & 0 & 0 & * & * & \cdots & * & * \\
\vdots && \vdots & \vdots & \vdots & \vdots & \ddots & \vdots & \vdots \\
0 && 0 & 0 & * & * & \cdots & x_{n-1} & y_{n-1} \\[1.5mm]
0 & \cdots & 0 & 0 & * & * & \cdots & y_{n-1} & (-s/q)y_{n-2} \\
\end{array} 
\right].
$}
\end{align*}
Then, by the definition of the matrix $W_{k}$ $(n_0 < k \leq n-2)$, we have
\begin{align*}
W(t)_{n-2}=
\scalebox{0.75}[0.75]
{$\left[
\begin{array}{cccccccccccccc}
0 & \cdots & \cdots & \cdots & \cdots & \cdots & \cdots & 0 & 0 \\
\vdots &&&&& \rotatebox[origin=c]{310}{\vdots} && \vdots & \vdots \\ 
\vdots &&&& qt & 0 & \cdots & 0 & 0 \\
\vdots &&& qt & 0 & 0 & \cdots & 0 & 0 \\
\vdots && qt & 0 & 0 & 0 & \cdots & 0 & 0 \\
\vdots & \rotatebox[origin=c]{310}{\vdots} & 0 & 0 & 0 & 0 & \cdots & 0 & 0 \\
\vdots && \vdots & \vdots & \vdots & \vdots & \ddots & \vdots & \vdots \\
0 && 0 & 0 & 0 & 0 & \cdots & x_{n-1} & y_{n-1} \\[1.5mm]
0 & \cdots & 0 & 0 & 0 & 0 & \cdots & y_{n-1} & (-s/q)y_{n-2} \\
\end{array} 
\right],
$}
\end{align*}
which completes the proof of  \cref{lem6.4} for even $n$.
\end{proof}
In the end, let us apply  \cref{lem6.4} to the matrix $A_{\bm c}(t)_1$ $(n \geq 5)$. Then, by the definition of the matrix $A_{\bm c}(t)_k$ $(2 \leq k \leq n-2)$ and  \cref{lem6.4}, we have
\begin{align} \label{eq6.4}
A_{\bm c}(t)_{n-2}
=
\scalebox{0.75}[0.75]
{$\left[
\begin{array}{cccccccc}
1 & 0 & \dots & \dots & \dots & 0 & 0 & 0 \\
0 &&&&& -(n-2)t & 0 & 0 \\
\vdots &&&&  \rotatebox[origin=c]{292}{\vdots} & \rotatebox[origin=c]{292}{\vdots} & 0 & 0 \\
\vdots &&& \rotatebox[origin=c]{292}{\vdots} & \hspace{3mm}  \rotatebox[origin=c]{292}{\vdots} & \rotatebox[origin=c]{292}{\vdots} & \vdots & \vdots \\ 
\vdots && \rotatebox[origin=c]{292}{\vdots} & \hspace{3mm}  \rotatebox[origin=c]{292}{\vdots} & \hspace{3mm}  \rotatebox[origin=c]{292}{\vdots} & \hspace{3mm}  \rotatebox[origin=c]{292}{\vdots} & \vdots & \vdots \\
0 & -(n-2)t & \rotatebox[origin=c]{292}{\vdots} & \rotatebox[origin=c]{292}{\vdots} & \hspace{3mm}  \rotatebox[origin=c]{292}{\vdots} && 0 & 0 \\ 
0 & 0 & 0 & \dots & \dots & 0 & a_{n-1,n-1}^{({\bm c})}(t)_{n-2} & a_{n-1,n}^{({\bm c})}(t)_{n-2} \\  
0 & 0 & 0 & \dots & \dots & 0 & a_{n,n-1}^{({\bm c})}(t)_{n-2} & a_{n,n}^{({\bm c})}(t)_{n-2}
\end{array}
\hspace{-0.5mm}\right],
$} 
\end{align}
\normalsize
where
\begin{align*}
\begin{cases}
a_{n-1,n-1}^{(\bm c)}(t)_{n-2}=2\left(1-2/n \right)t^2+\bar{x}_{n-1} \\[2mm]
a_{n-1,n}^{(\bm c)}(t)_{n-2}=a_{n,n-1}^{(\bm c)}(t)_{n-2}=\left(1-2/n\right)at^2-(nb\bar{x}_{n-2})/(n-2) \\[2mm]
a_{n,n}^{(\bm c)}(t)_{n-2}=\left\{\left(1-1/n\right)a^2-2b\right\}t^2+(n^2b^2\bar{x}_{n-3})/(n-2)^2.
\end{cases}
\end{align*}
Here, for any integer $m \geq 0$, we denote
\begin{align*}
&\bar{x}_{m}
=
\frac{(-1)^m\left(A^m-B^m\right)}{2^m\{-(n-2)\}^{m-2}\sqrt{(n-1)^2a^2-4n(n-2)b}}t, \\[2mm]
&A=-(n-1)a+\sqrt{(n-1)^2a^2-4n(n-2)b}, \\[2mm] 
&B=-(n-1)a-\sqrt{(n-1)^2a^2-4n(n-2)b}.
\end{align*}

\section{Proof of Theorems}
By a direct computation, we have
\begin{align} \label{eq6.5}
&\det
\left[
\begin{array}{cc}
a_{n-1,n-1}^{({\bm c})}(t)_{n-2} & a_{n-1,n}^{({\bm c})}(t)_{n-2} \\ \notag
a_{n,n-1}^{({\bm c})}(t)_{n-2} & a_{n,n}^{({\bm c})}(t)_{n-2}
\end{array}
\right] \\[1mm] \notag
&=\left(1-\frac{2}{n}\right)(a^2-4b)t^4 \\[1mm] \notag
&\hspace{2.5mm}+\frac{2(n-2)}{n}t^2\cdot\frac{n^2b^2\bar{x}_{n-3}}{(n-2)^2}+\left\{ \left( \frac{n-1}{n}a^2-2b \right) \right\}t^2\cdot\bar{x}_{n-1}+\frac{2(n-2)}{n}at^2\cdot\frac{nb\bar{x}_{n-2}}{n-2} \\[0mm] \notag
&\hspace{67.5mm}+\frac{n^2b^2}{(n-2)^2}\bar{x}_{n-1}\bar{x}_{n-3}-\frac{n^2b^2}{(n-2)^2}\bar{x}_{n-2}^2 \\[0mm]
&=\frac{(n-2)(a^2-4b)}{n}t^4+\left\{\frac{2nb^2\bar{x}_{n-3}}{n-2}+\frac{\left\{(n-1)a^2-2nb\right\}\bar{x}_{n-1}}{n}+2ab\bar{x}_{n-2}\right\}t^2 \\[0mm] \notag
&\hspace{76mm}+\frac{n^2b^2}{(n-2)^2}\left(\bar{x}_{n-1}\bar{x}_{n-3}-\bar{x}_{n-2}^2\right).
\end{align}
\normalsize
Here, let us put
\begin{align*}
&\alpha(t)=\left\{\frac{2nb^2\bar{x}_{n-3}}{n-2}+\frac{\left\{(n-1)a^2-2nb\right\}\bar{x}_{n-1}}{n}+2ab\bar{x}_{n-2}\right\}t^2, \\ 
&\beta(t)=\frac{n^2b^2}{(n-2)^2}\left(\bar{x}_{n-1}\bar{x}_{n-3}-\bar{x}_{n-2}^2\right).
\end{align*}
\begin{lem}\label{lem6.5}
Suppose $n\geq 5$ and put  
\begin{align*}
S_k=&(n-1)^3\binom{n-k-3}{k}a^4-\frac{n(n-1)\left\{ 5n^2-(6k+23)n+10k+24 \right\}}{n-k-3}\binom{n-k-3}{k}a^2b \\
&\hspace{80mm}+4n^2(n-2)\binom{n-k-4}{k}b^2. 
\end{align*}
\normalsize
Then, we have
\begin{align*}
&\alpha(t)=
\begin{cases}
\displaystyle \sum_{k=0}^{m_0}\frac{(-1)^{n+k}n^k(n-1)^{n-2k-4}(n-2)^ka^{n-2k-4}b^kS_k}{n(n-2)^{n-3}}t^3 & \hspace{-1mm} \text{$(a \neq 0$, or $a=0$, $n : $ even$)$}, \\
0 & \hspace{-1mm} \text{$(a=0$, $n : $ odd$)$}.
\end{cases} \\
&\beta(t)=-\frac{n^{n-1}b^{n-1}}{(n-2)^{n-3}}t^2.
\end{align*}
\end{lem}
To carry out these computations, we need some combinatorial identities.
\begin{lem}{\cite[Chapter 2, Problem 18 (c), Chapter 6, Problem 18 (a)]{rio}}\label{lem6.6}
Let $N$ $(\geq 0)$ be an integer and put $\underline{m}=\lfloor N/2 \rfloor$, $\overline{m}=\lceil N/2 \rceil=\lfloor (N+1)/2 \rfloor$. Then, we have
\begin{align*}
\sum_{j=k}^{\underline{m}}\binom{N+1}{2j+1}\binom{j}{k}&=2^{N-2k}\binom{N-k}{k} \hspace{35mm} (k \in \Z, \ 0 \leq k \leq \underline{m}), \\[2mm] 
\sum_{j=k}^{\overline{m}}\binom{N+1}{2j}\binom{j}{k}&=2^{N-2k}\left[2\binom{N+1-k}{k}-\binom{N-k}{k} \right] \\[2mm]
&=2^{N-2k}\left[\binom{N+1-k}{k}+\binom{N-k}{k-1} \right] \hspace{5mm} (k \in \Z, \ 0 \leq k \leq \overline{m}). 
\end{align*}
\end{lem}
\begin{proof}
We omit the proof of the first identity and let us prove the second one. By using the convention
\begin{align*}
\binom{n}{n+m}=0 \hspace{3mm} (n \in \Z_{\geq 0}, \ m \in \Z_{\geq 1} ),
\end{align*}
we have
\begin{align*}
\sum_{j=k}^{\underline{m}}\binom{N+1}{2j+1}\binom{j}{k}=\sum_{j=k}^{\overline{m}}\binom{N+1}{2j+1}\binom{j}{k}
\end{align*}
and hence
\begin{align*}
\sum_{j=k}^{\overline{m}}\binom{N+1}{2j}\binom{j}{k}+\sum_{j=k}^{\underline{m}}\binom{N+1}{2j+1}\binom{j}{k}
&=\sum_{j=k}^{\overline{m}}\binom{N+1}{2j}\binom{j}{k}+\sum_{j=k}^{\overline{m}}\binom{N+1}{2j+1}\binom{j}{k} \\
&=\sum_{j=k}^{\overline{m}}\left\{ \binom{N+1}{2j}+\binom{N+1}{2j+1} \right\}\binom{j}{k} \\
&=\sum_{j=k}^{\overline{m}}\binom{N+2}{2j+1}\binom{j}{k}
\end{align*}
for any $k$ $(0 \leq k \leq \overline{m})$. Thus, by using the first identity, we get the second one.
\end{proof}

\begin{proof}[Proof of  \cref{lem6.5}]
We first prove the second equality. Then, by the definition of $\bar{x}_m$, we have
\begin{align*}
\beta(t)&=\frac{n^2b^2}{(n-2)^2}\cdot\frac{-A^{n-1}B^{n-3}-A^{n-3}B^{n-1}+2A^{n-2}B^{n-2}}{2^{2n-4}(n-2)^{2n-8}\left\{(n-1)^2a^2-4n(n-2)b\right\}}t^2 \\[2mm]
&=-\frac{n^2b^2\left\{ \left( 4n(n-2)b \right)^{n-3}\left(A^2+B^2 \right)-2\left( 4n(n-2)b \right)^{n-2} \right\}}{2^{2n-4}(n-2)^{2n-6}\left\{(n-1)^2a^2-4n(n-2)b\right\}}t^2 \\[2mm]
&=-\frac{2^{2n-6}n^{n-1}(n-2)^{n-3}b^{n-1}\cdot4\left\{(n-1)^2a^2-4n(n-2)b\right\}}{2^{2n-4}(n-2)^{2n-6}\left\{(n-1)^2a^2-4n(n-2)b\right\}}t^2 \\[2mm]
&=-\frac{n^{n-1}b^{n-1}}{(n-2)^{n-3}}t^2,
\end{align*}
which is the second equality. For the first equality, let us first suppose $a=0$. Then, we have  
\begin{align*}
\bar{x}_m&=\frac{(-1)^m\left\{\left\{-4n(n-2)b\right\}^{m/2}-(-1)^m\left\{-4n(n-2)b\right\}^{m/2}\right\}}{2^m\left\{-(n-2)\right\}^{m-2}\left\{-4n(n-2)b\right\}^{1/2}}t \\
&=
\begin{cases}
\displaystyle \frac{2\left\{-4n(n-2)b\right\}^{(m-1)/2}}{2^m(n-2)^{m-2}}t & \text{($m : $ odd)}, \\
0 & \text{($m : $ even)} 
\end{cases} \\
&=
\begin{cases}
\displaystyle \frac{(-nb)^{(m-1)/2}}{(n-2)^{(m-3)/2}}t & \text{($m : $ odd)}, \\
0 & \text{($m : $ even)} 
\end{cases}
\end{align*}
and hence we have
\begin{align*}
\alpha(t)&=\left\{ \frac{2nb^2\bar{x}_{n-3}}{n-2}-2b\bar{x}_{n-1} \right\}t^2 \\
&=
\begin{cases}
\displaystyle \left\{ \frac{2nb^2}{n-2}\cdot\frac{(-nb)^{(n-4)/2}}{(n-2)^{(n-6)/2}}t-\frac{2b(-nb)^{(n-2)/2}}{(n-2)^{(n-4)/2}}t \right\}t^2 & \text{($n : $ even)}, \\
0 & \text{($n : $ odd)}
\end{cases} \\
&=
\begin{cases}
\displaystyle \frac{(-1)^{n/2}4n^{(n-2)/2}b^{n/2}}{(n-2)^{(n-4)/2}}t^3 & \text{($n : $ even)}, \\
0 & \text{($n : $ odd)},
\end{cases}
\end{align*}
which is the claim of   \cref{lem6.5} for the case $a=0$ since for even $n$ $(n \geq 6)$, 
\begin{align*}
&\displaystyle \frac{\sum_{k=0}^{m_0}(-1)^{n+k}n^k(n-1)^{n-2k-4}(n-2)^ka^{n-2k-4}b^kS_k}{n(n-2)^{n-3}}t^3 \\
&=\frac{(-1)^{(3n-4)/2}n^{(n-4)/2}(n-2)^{(n-4)/2}b^{(n-4)/2}\cdot4n^2(n-2)b^2}{n(n-2)^{n-3}}t^3 \\
&=\frac{(-1)^{n/2}4n^{(n-2)/2}b^{n/2}}{(n-2)^{(n-4)/2}}t^3.
\end{align*}
Next, we suppose $a\neq0$ and let us put 
\begin{align*}
&F=(n-1)^2a^3-2n(2n-3)ab, \ G=(n-1)a^2-2nb, \\ 
&H=(n-1)^2a^2-4n(n-2)b.
\end{align*}
Then, we have  
\normalsize
\begin{align*}
&\frac{2nb^2\bar{x}_{n-3}}{n-2}+\frac{\left\{(n-1)a^2-2nb\right\}\bar{x}_{n-1}}{n}+2ab\bar{x}_{n-2} \\[1mm] 
&=\frac{2nb^2}{n-2}\frac{\left(A^{n-3}-B^{n-3}\right)}{2^{n-3}(n-2)^{n-5}\sqrt{H}}t+\frac{(n-1)a^2-2nb}{n}\frac{\left(A^{n-1}-B^{n-1}\right)}{2^{n-1}(n-2)^{n-3}\sqrt{H}}t \\[2mm]
&\hspace{90mm}+\frac{2ab(A^{n-2}-B^{n-2})}{2^{n-2}(n-2)^{n-4}\sqrt{H}}t \\[2mm] 
&=\frac{8n^2(n-2)b^2(A^{n-3}-B^{n-3})}{2^{n-1}n(n-2)^{n-3}\sqrt{H}}t+\frac{\left\{ (n-1)a^2-2nb \right\}(A^{n-1}-B^{n-1})}{2^{n-1}n(n-2)^{n-3}\sqrt{H}}t \\[1mm] 
&\hspace{78mm}+\frac{4n(n-2)ab(A^{n-2}-B^{n-2})}{2^{n-1}n(n-2)^{n-3}\sqrt{H}}t \\[1mm] 
&=\frac{8n^2(n-2)b^2+\left\{ (n-1)a^2-2nb \right\}A^2+4n(n-2)abA}{2^{n-1}n(n-2)^{n-3}\sqrt{H}}A^{n-3}t \\[1mm] 
&\hspace{28mm}-\frac{8n^2(n-2)b^2+\left\{ (n-1)a^2-2nb \right\}B^2+4n(n-2)abB}{2^{n-1}n(n-2)^{n-3}\sqrt{H}}B^{n-3}t
\end{align*}
\normalsize
and
\normalsize
\begin{align*}
&8n^2(n-2)b^2+\left\{ (n-1)a^2-2nb \right\}A^2+4n(n-2)abA \\
&=8n^2(n-2)b^2+\left\{ (n-1)a^2-2nb \right\}\left\{ 2(n-1)^2a^2-4n(n-2)b-2(n-1)a\sqrt{H} \right\} \\
&\hspace{70mm}+4n(n-2)ab\left\{ -(n-1)a+\sqrt{H} \right\} \\
&=-2\left\{ (n-1)^2a^3-2n(2n-3)ab \right\}\sqrt{H}+2\left\{ (n-1)a^2-2nb \right\}H \\
&=-2F\sqrt{H}+2GH, \\[3mm]
&8n^2(n-2)b^2+\left\{ (n-1)a^2-2nb \right\}B^2+4n(n-2)abB \\
&=8n^2(n-2)b^2+\left\{ (n-1)a^2-2nb \right\}\left\{ 2(n-1)^2a^2-4n(n-2)b+2(n-1)a\sqrt{H} \right\} \\
&\hspace{70mm}+4n(n-2)ab\left\{ -(n-1)a-\sqrt{H} \right\} \\
&=2\left\{ (n-1)^2a^3-2n(2n-3)ab \right\}\sqrt{H}+2\left\{ (n-1)a^2-2nb \right\}H, \\
&=2F\sqrt{H}+2GH.
\end{align*}
\normalsize
Hence, by putting
\begin{align*}
A^{n-3}=\left\{-(n-1)a+\sqrt{(n-1)^2a^2-4n(n-2)b} \right\}^{n-3}=I+J\sqrt{H}, 
\end{align*}
we have
\begin{align*}
&\frac{2nb^2\bar{x}_{n-3}}{n-2}+\frac{\left\{(n-1)a^2-2nb\right\}\bar{x}_{n-1}}{n}+2ab\bar{x}_{n-2} \\[1mm] 
&=\frac{\left\{-2F\sqrt{H}+2GH\right\}\left( I+J\sqrt{H} \right)-\left\{2F\sqrt{H}+2GH\right\}\left( I-J\sqrt{H} \right)}{2^{n-1}n(n-2)^{n-3}\sqrt{H}}t \\
&=\frac{-FI+GHJ}{2^{n-3}n(n-2)^{n-3}}t.
\end{align*}
Then, by using   \cref{lem6.6}, we have
\begin{align*}
I&=\sum_{\ell=0}^{m_0}\binom{n-3}{2\ell}\left\{ -(n-1)a \right\}^{n-2\ell-3}\left\{ (n-1)^2a^2-4n(n-2)b \right\}^{\ell} \\
&=\sum_{\ell=0}^{m_0}\binom{n-3}{2\ell}\left\{ -(n-1)a \right\}^{n-2\ell-3}\left\{\sum_{k=0}^{\ell}\binom{\ell}{k}\left\{ (n-1)^2a^2 \right\}^{\ell-k}\left\{ -4n(n-2)b \right\}^{k}\right\} \\
&=\sum_{k=0}^{m_0}\sum_{\ell=k}^{m_0}\binom{n-3}{2\ell}\binom{\ell}{k}(-1)^{n+k+1}2^{2k}n^k(n-1)^{n-2k-3}(n-2)^ka^{n-2k-3}b^k \\
&=\sum_{k=0}^{m_0}2^{n-2k-4}\left\{\binom{n-k-3}{k}+\binom{n-k-4}{k-1} \right\} \\
&\hspace{43mm}\times (-1)^{n+k+1}2^{2k}n^k(n-1)^{n-2k-3}(n-2)^ka^{n-2k-3}b^k \\
&=\sum_{k=0}^{m_0} (-1)^{n+k+1}2^{n-4}n^k(n-1)^{n-2k-3}(n-2)^k\frac{n-3}{n-k-3}\binom{n-k-3}{k}a^{n-2k-3}b^k 
\end{align*}
and
\begin{align*}
J&=\sum_{\ell=0}^{m_0}\binom{n-3}{2\ell+1}\left\{ -(n-1)a \right\}^{n-2\ell-4}\left\{ (n-1)^2a^2-4n(n-2)b \right\}^{\ell} \\
&=\sum_{\ell=0}^{m_0}\binom{n-3}{2\ell+1}\left\{ -(n-1)a \right\}^{n-2\ell-4}\left\{\sum_{k=0}^{\ell}\binom{\ell}{k}\{ (n-1)^2a^2 \}^{\ell-k}\left\{ -4n(n-2)b \right\}^k\right\} \\
&=\sum_{k=0}^{m_0}\sum_{\ell=k}^{m_0}\binom{n-3}{2\ell+1}\binom{\ell}{k}(-1)^{n+k}2^{2k}n^k(n-1)^{n-2k-4}(n-2)^ka^{n-2k-4}b^k \\
&=\sum_{k=0}^{m_0}2^{n-2k-4}\binom{n-k-4}{k}(-1)^{n+k}2^{2k}n^k(n-1)^{n-2k-4}(n-2)^ka^{n-2k-4}b^k \\
&=\sum_{k=0}^{m_0}(-1)^{n+k}2^{n-4}n^k(n-1)^{n-2k-4}(n-2)^k\binom{n-k-4}{k}a^{n-2k-4}b^k,
\end{align*}
which implies, 
\begin{align*}
&-FI+GHJ \\
&=-\{(n-1)^2a^3-2n(2n-3)ab\}I+\{ (n-1)^3a^4-2n(n-1)(3n-5)a^2b+8n^2(n-2)b^2 \}J \\
&=\sum_{k=0}^{m_0}(-1)^{n+k}2^{n-4}n^k(n-1)^{n-2k-4}(n-2)^ka^{n-2k-4}b^k \\
&\hspace{5mm}\times \biggr\{ (n-1)^3\frac{n-3}{n-k-3}\binom{n-k-3}{k}a^4-2n(n-1)(2n-3)\frac{n-3}{n-k-3}\binom{n-k-3}{k}a^2b \\
&\hspace{8mm}+(n-1)^3\binom{n-k-4}{k}a^4-2n(n-1)(3n-5)\binom{n-k-4}{k}a^2b+8n^2(n-2)\binom{n-k-4}{k}b^2 \biggl\} \\
&=\sum_{k=0}^{m_0}(-1)^{n+k}2^{n-3}n^k(n-1)^{n-2k-4}(n-2)^ka^{n-2k-4}b^k \\
&\hspace{5mm}\times \biggr\{ (n-1)^3\binom{n-k-3}{k}a^4-\frac{n(n-1)\left\{ 5n^2-(6k+23)n+10k+24 \right\}}{n-k-3}\binom{n-k-3}{k}a^2b \\
&\hspace{88mm}+4n^2(n-2)\binom{n-k-4}{k}b^2 \biggl\}
\end{align*}
\normalsize
Therefore, we finally obtain
\begin{align*}
\alpha(t)&=\left\{\frac{2nb^2\bar{x}_{n-3}}{n-2}+\frac{\left\{(n-1)a^2-2nb\right\}\bar{x}_{n-1}}{n}+2ab\bar{x}_{n-2}\right\}t^2 =\frac{-FI+GHJ}{2^{n-3}n(n-2)^{n-3}}t^3 \\
&=\sum_{k=0}^{m_0}\frac{(-1)^{n+k}n^k(n-1)^{n-2k-4}(n-2)^ka^{n-2k-4}b^kS_k}{n(n-2)^{n-3}}t^3,
\end{align*}
\normalsize
where
\begin{align*}
S_k&=(n-1)^3\binom{n-k-3}{k}a^4-\frac{n(n-1)\left\{ 5n^2-(6k+23)n+10k+24 \right\}}{n-k-3}\binom{n-k-3}{k}a^2b \\
&\hspace{81mm}+4n^2(n-2)\binom{n-k-4}{k}b^2,
\end{align*}
\normalsize
which completes the proof of \cref{lem6.5}.
\end{proof}

\begin{proof}[Proof of  \cref{thm6.0}]
Theorem 3 has been proved for $n=3, 4$ and hence let us assume $n \geq 5$. Then, by the definition of $A_{\bm c}(t)_{n-2}$ and equations \eqref{eq6.4}, \eqref{eq6.5}, we have
\begin{align*}
&\Delta\left(f_{\bm c}(t; x)\right)=\det A_{\bm c}(t)=n\cdot\det A_{\bm c}(t)_{n-2} \\
&=n\cdot(-1)^{\lceil (n-3)/2 \rceil}(n-2)^{n-3}t^{n-3}\cdot
\det
\left[
\begin{array}{cc}
a_{n-1,n-1}^{({\bm c})}(t)_{n-2} & a_{n-1,n}^{({\bm c})}(t)_{n-2} \\ \notag
a_{n,n-1}^{({\bm c})}(t)_{n-2} & a_{n,n}^{({\bm c})}(t)_{n-2}
\end{array}
\right].
\end{align*}
Therefore, by  \cref{lem6.5}, we have
\begin{align*}
\Delta\left(f_{\bm c}(t; x)\right)=(-1)^{m_1}t^{n-1}\Biggl\{(n-2)^{n-2}(a^2-4b)t^2+\gamma_{\bm c}t-n^nb^{n-1} \Biggr\}, 
\end{align*}
where
\begin{align*}
\gamma_{\bm c}=
\begin{cases}
\displaystyle \sum_{k=0}^{m_0}(-1)^{n+k}n^k(n-1)^{n-2k-4}(n-2)^ka^{n-2k-4}b^kS_k & \hspace{-1mm} \text{$(a \neq 0$,  or $a=0$, $n : $ even$)$}, \\
0 & \hspace{-1mm} \text{$(a=0$, $n : $ odd$)$},
\end{cases}
\end{align*}
\normalsize
which completes the proof.
\end{proof}

Now we are ready to prove \cref{thm6.1}.

\begin{proof}[Proof of  \cref{thm6.1}]
Put
$Q(t)=(n-2)^{n-2}(a^2-4b)t^2+\gamma_{\bm c}t-n^nb^{n-1}$
and suppose $b=(n-1)^2a^2/4n(n-2)$. 
Then, by equation \eqref{eq6.5} and  \cref{lem6.5}, we have
\begin{align*}
&Q(t) \\
&=n(n-2)^{n-3}\Biggl\{\frac{(n-2)(a^2-4b)}{n}t^2+\left\{\frac{2nb^2\bar{x}_{n-3}}{n-2}+\frac{\left\{(n-1)a^2-2nb\right\}\bar{x}_{n-1}}{n}+2ab\bar{x}_{n-2}\right\} \\[1mm] 
&\hspace{105mm}-\frac{n^{n-1}b^{n-1}}{(n-2)^{n-3}} \Biggr\} \\
&=(n-2)^{n-2}\left\{a^2-\frac{(n-1)^2a^2}{n(n-2)}\right\}t^2 \\
&\hspace{18mm}+n(n-2)^{n-3}\left\{\frac{2nb^2\bar{x}_{n-3}}{n-2}+\frac{\left\{(n-1)a^2-2nb\right\}\bar{x}_{n-1}}{n}+2ab\bar{x}_{n-2}\right\}-n^nb^{n-1} \\
&=-\frac{(n-2)^{n-3}a^2}{n}t^2 \\
&\hspace{18mm}+n(n-2)^{n-3}\left\{\frac{2nb^2\bar{x}_{n-3}}{n-2}+\frac{\left\{(n-1)a^2-2nb\right\}\bar{x}_{n-1}}{n}+2ab\bar{x}_{n-2}\right\}-n^nb^{n-1}.
\end{align*}
\normalsize
Here, by equation \eqref{eq6.3}, we have
\begin{align*}
&n(n-2)^{n-3}\left\{\frac{2nb^2\bar{x}_{n-3}}{n-2}+\frac{\left\{(n-1)a^2-2nb\right\}\bar{x}_{n-1}}{n}+2ab\bar{x}_{n-2}\right\} \\
&=n(n-2)^{n-3}\Biggl\{\frac{2n}{n-2}\cdot\frac{(n-1)^4a^4}{16n^2(n-2)^2}\cdot\frac{(-1)^{n-3}(n-3)\{-(n-1)a\}^{n-4}}{2^{n-4}\left\{ -(n-2) \right\}^{n-5}}t \\
&\hspace{30mm}+\left\{\frac{(n-1)a^2}{n}-\frac{(n-1)^2a^2}{2n(n-2)}\right\}\cdot\frac{(-1)^{n-1}(n-1)\{-(n-1)a\}^{n-2}}{2^{n-2}\left\{ -(n-2) \right\}^{n-3}}t \\
&\hspace{59mm}+\frac{(n-1)^2a^3}{2n(n-2)}\cdot\frac{(-1)^{n-2}(n-2)\{-(n-1)a\}^{n-3}}{2^{n-3}\left\{ -(n-2) \right\}^{n-4}}t\Biggr\} \\
&=n(n-2)^{n-3}\left\{\frac{(-1)^n(n-1)^n(n-3)a^n}{2^{n-1}n(n-2)^{n-2}}t+\frac{(-1)^n(n-1)^n(n-3)a^n}{2^{n-1}n(n-2)^{n-2}}t-\frac{(-1)^n(n-1)^{n-1}a^n}{2^{n-2}n(n-2)^{n-4}}t\right\} \\
&=-\frac{(-1)^n(n-1)^{n-1}a^n}{2^{n-2}(n-2)}t,
\end{align*}
\normalsize
which implies
\begin{align*}
Q(t)&=-\frac{(n-2)^{n-3}a^2}{n}t^2-\frac{(-1)^n(n-1)^{n-1}a^n}{2^{n-2}(n-2)}t-n^nb^{n-1}, \\
\Delta\left(Q(t)\right)&=\left\{ -\frac{(-1)^n(n-1)^{n-1}a^n}{2^{n-2}(n-2)} \right\}^2-4\cdot\left\{ -\frac{(n-2)^{n-3}a^2}{n} \right\}\cdot\left( -n^nb^{n-1}\right) \\
&=\frac{(n-1)^{2n-2}a^{2n}}{2^{2n-4}(n-2)^2}-\frac{4(n-2)^{n-3}a^2}{n}\cdot n^{n} \cdot\left\{\frac{(n-1)^2a^2}{4n(n-2)}\right\}^{n-1}=0
\end{align*}
and hence, by completing the square, we have
\begin{align*}
Q(t)=-\frac{(n-2)^{n-3}a^2}{n}\left\{t+\frac{(-1)^nn(n-1)^{n-1}a^{n-2}}{2^{n-1}(n-2)^{n-2}} \right\}^2.
\end{align*}
Thus, if $n$ is even, we have $\alpha_{\bm c}=0$ and hence by \cref{thm5.2}, we have $N_{f_{\bm c}(t; x)}=0$ for any positive real number $t$ since $x^2+ax+b$ has no real root in this case. Then, by considering the graph of the function $y=x^n+t(x^2+ax+b)$ $(t>0)$, we can conclude that $N_{f_{\bm c}(t; x)}=0$ whenever $(n-1)^2a^2/4n(n-2)\leq b$. 
\end{proof}


\begin{rem}
It was pointed out to us by J. Gutierrez that the proof of the Cor.~1, it is elementary and all the machinery of \cref{thm6.0} is not needed. However, we decided to present it here with the intention that perhaps such techniques can be generalized to polynomials $f(x)=x^n + t\cdot g(x)$, $\deg g \geq 3$. 
\end{rem}

\nocite{*}

\bibliographystyle{amsalpha} 

\bibliography{ref}{}

\begin{bibdiv}
\begin{biblist}

\bib{shimol}{article}{
      author={Ben-Shimol, Oz},
       title={On {G}alois groups of prime degree polynomials with complex
  roots},
        date={2009},
        ISSN={1726-3255},
     journal={Algebra Discrete Math.},
      number={2},
       pages={99\ndash 107},
      review={\MR{2589076}},
}

\bib{B-B-G}{article}{
      author={Beresnevich, Victor},
      author={Bernik, Vasili},
      author={G\"otze, Friedrich},
       title={Integral polynomials with small discriminants and resultants},
        date={2016},
        ISSN={0001-8708},
     journal={Adv. Math.},
      volume={298},
       pages={393\ndash 412},
         url={https://doi.org/10.1016/j.aim.2016.04.022},
      review={\MR{3505745}},
}

\bib{MR3782461}{incollection}{
      author={Beshaj, L.},
      author={Hidalgo, R.},
      author={Kruk, S.},
      author={Malmendier, A.},
      author={Quispe, S.},
      author={Shaska, T.},
       title={Rational points in the moduli space of genus two},
        date={2018},
   booktitle={Higher genus curves in mathematical physics and arithmetic
  geometry},
      series={Contemp. Math.},
      volume={703},
   publisher={Amer. Math. Soc., Providence, RI},
       pages={83\ndash 115},
         url={https://doi.org/10.1090/conm/703/14132},
      review={\MR{3782461}},
}

\bib{MR2755393}{article}{
      author={Beshaj, Lubjana},
       title={Singular locus on the space of genus 2 curves with decomposable
  {J}acobians},
        date={2010},
        ISSN={1930-1235},
     journal={Albanian J. Math.},
      volume={4},
      number={4},
       pages={147\ndash 160},
      review={\MR{2755393}},
}

\bib{reduction}{incollection}{
      author={Beshaj, Lubjana},
       title={Reduction theory of binary forms},
        date={2015},
   booktitle={Advances on superelliptic curves and their applications},
      series={NATO Sci. Peace Secur. Ser. D Inf. Commun. Secur.},
      volume={41},
   publisher={IOS, Amsterdam},
       pages={84\ndash 116},
      review={\MR{3525574}},
}

\bib{beshaj-2}{incollection}{
      author={Beshaj, Lubjana},
       title={Reduction theory of binary forms},
        date={2015},
   booktitle={Advances on superelliptic curves and their applications},
      series={NATO Sci. Peace Secur. Ser. D Inf. Commun. Secur.},
      volume={41},
   publisher={IOS, Amsterdam},
       pages={84\ndash 116},
      review={\MR{3525574}},
}

\bib{beshaj-1}{book}{
      author={Beshaj, Lubjana},
       title={Integral binary forms with minimal height},
   publisher={ProQuest LLC, Ann Arbor, MI},
        date={2016},
        ISBN={978-1369-23527-2},
  url={http://gateway.proquest.com/openurl?url_ver=Z39.88-2004&rft_val_fmt=info:ofi/fmt:kev:mtx:dissertation&res_dat=xri:pqm&rft_dat=xri:pqdiss:10169331},
        note={Thesis (Ph.D.)--Oakland University},
      review={\MR{3579531}},
}

\bib{MR3782460}{incollection}{
      author={Beshaj, Lubjana},
       title={Minimal integral {W}eierstrass equations for genus 2 curves},
        date={2018},
   booktitle={Higher genus curves in mathematical physics and arithmetic
  geometry},
      series={Contemp. Math.},
      volume={703},
   publisher={Amer. Math. Soc., Providence, RI},
       pages={63\ndash 82},
         url={https://doi.org/10.1090/conm/703/14131},
      review={\MR{3782460}},
}

\bib{b-sh}{incollection}{
      author={Bialostocki, A.},
      author={Shaska, T.},
       title={Galois groups of prime degree polynomials with nonreal roots},
        date={2005},
   booktitle={Computational aspects of algebraic curves},
      series={Lecture Notes Ser. Comput.},
      volume={13},
   publisher={World Sci. Publ., Hackensack, NJ},
       pages={243\ndash 255},
         url={http://dx.doi.org/10.1142/9789812701640_0015},
      review={\MR{2182043}},
}

\bib{MR3303602}{article}{
      author={Boyd, David~W.},
      author={Martin, Greg},
      author={Thom, Mark},
       title={Squarefree values of trinomial discriminants},
        date={2015},
        ISSN={1461-1570},
     journal={LMS J. Comput. Math.},
      volume={18},
      number={1},
       pages={148\ndash 169},
         url={https://doi.org/10.1112/S1461157014000436},
      review={\MR{3303602}},
}

\bib{D-S}{article}{
      author={Dilcher, Karl},
      author={Stolarsky, Kenneth~B.},
       title={Resultants and discriminants of {C}hebyshev and related
  polynomials},
        date={2005},
        ISSN={0002-9947},
     journal={Trans. Amer. Math. Soc.},
      volume={357},
      number={3},
       pages={965\ndash 981},
         url={https://doi.org/10.1090/S0002-9947-04-03687-6},
      review={\MR{2110427}},
}

\bib{MR2247810}{article}{
      author={Finch, Carrie},
      author={Jones, Lenny},
       title={On the irreducibility of {$\{-1,0,1\}$}-quadrinomials},
        date={2006},
        ISSN={1553-1732},
     journal={Integers},
      volume={6},
       pages={A16, 4},
      review={\MR{2247810}},
}

\bib{MR3194146}{article}{
      author={Flammang, V.},
       title={The {M}ahler measure of trinomials of height 1},
        date={2014},
        ISSN={1446-7887},
     journal={J. Aust. Math. Soc.},
      volume={96},
      number={2},
       pages={231\ndash 243},
         url={https://doi.org/10.1017/S1446788713000633},
      review={\MR{3194146}},
}

\bib{f-sh}{inproceedings}{
      author={Frey, G.},
      author={Shaska, T.},
       title={Abelian variaties and cryptography},
        date={2018},
   booktitle={Algebraic curves and their applications},
   publisher={American Mathematical Society},
}

\bib{fro}{article}{
      author={Frobenius, G.},
       title={Ueber das {T}r\"agheitsgesetz der quadratischen {F}ormen},
        date={1895},
        ISSN={0075-4102},
     journal={J. Reine Angew. Math.},
      volume={114},
       pages={187\ndash 230},
         url={https://doi.org/10.1515/crll.1895.114.187},
      review={\MR{1580376}},
}

\bib{fuh}{book}{
      author={Fuhrmann, Paul~A.},
       title={A polynomial approach to linear algebra},
     edition={Second},
      series={Universitext},
   publisher={Springer, New York},
        date={2012},
        ISBN={978-1-4614-0337-1},
         url={https://doi.org/10.1007/978-1-4614-0338-8},
      review={\MR{2894784}},
}

\bib{g-d}{article}{
      author={Greenfield, Gary~R.},
      author={Drucker, Daniel},
       title={On the discriminant of a trinomial},
        date={1984},
        ISSN={0024-3795},
     journal={Linear Algebra Appl.},
      volume={62},
       pages={105\ndash 112},
         url={https://doi.org/10.1016/0024-3795(84)90089-2},
      review={\MR{761061}},
}

\bib{MR3782459}{incollection}{
      author={Hidalgo, Ruben},
      author={Shaska, Tony},
       title={On the field of moduli of superelliptic curves},
        date={2018},
   booktitle={Higher genus curves in mathematical physics and arithmetic
  geometry},
      series={Contemp. Math.},
      volume={703},
   publisher={Amer. Math. Soc., Providence, RI},
       pages={47\ndash 62},
         url={https://doi.org/10.1090/conm/703/14130},
      review={\MR{3782459}},
}

\bib{MR2646435}{article}{
      author={Jankauskas, Jonas},
       title={On the reducibility of certain quadrinomials},
        date={2010},
        ISSN={0017-095X},
     journal={Glas. Mat. Ser. III},
      volume={45(65)},
      number={1},
       pages={31\ndash 41},
         url={https://doi.org/10.3336/gm.45.1.03},
      review={\MR{2646435}},
}

\bib{j-sh}{article}{
      author={Joyner, David},
      author={Shaska, Tony},
       title={Self-inversive polynomials, curves, and codes},
        date={201606},
      eprint={1606.03159},
         url={https://arxiv.org/abs/1606.03159},
}

\bib{MR3782467}{incollection}{
      author={Joyner, David},
      author={Shaska, Tony},
       title={Self-inversive polynomials, curves, and codes},
        date={2018},
   booktitle={Higher genus curves in mathematical physics and arithmetic
  geometry},
      series={Contemp. Math.},
      volume={703},
   publisher={Amer. Math. Soc., Providence, RI},
       pages={189\ndash 208},
         url={https://doi.org/10.1090/conm/703/14138},
      review={\MR{3782467}},
}

\bib{Ked}{article}{
      author={Kedlaya, Kiran~S.},
       title={A construction of polynomials with squarefree discriminants},
        date={2012},
        ISSN={0002-9939},
     journal={Proc. Amer. Math. Soc.},
      volume={140},
      number={9},
       pages={3025\ndash 3033},
         url={https://doi.org/10.1090/S0002-9939-2012-11231-6},
      review={\MR{2917075}},
}

\bib{MR0124313}{article}{
      author={Ljunggren, Wilhelm},
       title={On the irreducibility of certain trinomials and quadrinomials},
        date={1960},
        ISSN={0025-5521},
     journal={Math. Scand.},
      volume={8},
       pages={65\ndash 70},
         url={https://doi.org/10.7146/math.scand.a-10593},
      review={\MR{0124313}},
}

\bib{MR3712162}{article}{
      author={Malmendier, A.},
      author={Shaska, T.},
       title={The {S}atake sextic in {F}-theory},
        date={2017},
        ISSN={0393-0440},
     journal={J. Geom. Phys.},
      volume={120},
       pages={290\ndash 305},
         url={https://doi.org/10.1016/j.geomphys.2017.06.010},
      review={\MR{3712162}},
}

\bib{MR3731039}{article}{
      author={Malmendier, Andreas},
      author={Shaska, Tony},
       title={A universal genus-two curve from {S}iegel modular forms},
        date={2017},
        ISSN={1815-0659},
     journal={SIGMA Symmetry Integrability Geom. Methods Appl.},
      volume={13},
       pages={Paper No. 089, 17},
         url={https://doi.org/10.3842/SIGMA.2017.089},
      review={\MR{3731039}},
}

\bib{w-height}{article}{
      author={Mandili, J.},
      author={Shaska, T.},
       title={Heights on weighted projective spaces},
        date={2018},
      eprint={1801.06250},
         url={https://arxiv.org/pdf/1801.06250},
}

\bib{MR815429}{article}{
      author={Mills, W.~H.},
       title={The factorization of certain quadrinomials},
        date={1985},
        ISSN={0025-5521},
     journal={Math. Scand.},
      volume={57},
      number={1},
       pages={44\ndash 50},
         url={https://doi.org/10.7146/math.scand.a-12105},
      review={\MR{815429}},
}

\bib{Mori}{article}{
      author={Mori, Shigefumi},
       title={The endomorphism rings of some abelian varieties. {II}},
        date={1977},
     journal={Japan. J. Math. (N.S.)},
      volume={3},
      number={1},
       pages={105\ndash 109},
      review={\MR{529440}},
}

\bib{otake}{article}{
      author={Otake, Shuichi},
       title={Counting the number of distinct real roots of certain polynomials
  by {B}ezoutian and the {G}alois groups over the rational number field},
        date={2013},
        ISSN={0308-1087},
     journal={Linear Multilinear Algebra},
      volume={61},
      number={4},
       pages={429\ndash 441},
         url={http://dx.doi.org/10.1080/03081087.2012.689983},
      review={\MR{3005628}},
}

\bib{MR3314338}{article}{
      author={Otake, Shuichi},
       title={A {B}ezoutian approach to orthogonal decompositions of trace
  forms or integral trace forms of some classical polynomials},
        date={2015},
        ISSN={0024-3795},
     journal={Linear Algebra Appl.},
      volume={471},
       pages={291\ndash 319},
         url={http://dx.doi.org/10.1016/j.laa.2015.01.005},
      review={\MR{3314338}},
}

\bib{o-sh-3}{inproceedings}{
      author={Otake, Shuichi},
      author={Shaska, Tony},
       title={On the discriminant of a certain quartinomials},
        date={2018},
   booktitle={Algebraic curves and their applications},
      editor={Beshaj, L.},
       pages={55\ndash 61},
}

\bib{o-sh}{article}{
      author={Otake, Shuichi},
      author={Shaska, Tony},
       title={Some remarks on the non-real roots of polynomials},
        date={201802},
      eprint={1802.02708},
         url={https://arxiv.org/abs/1802.02708},
}

\bib{rio}{book}{
      author={Riordan, John},
       title={Combinatorial identities},
   publisher={John Wiley \& Sons, Inc., New York-London-Sydney},
        date={1968},
      review={\MR{0231725}},
}

\bib{Se}{article}{
      author={Selmer, Ernst~S.},
       title={On the irreducibility of certain trinomials},
        date={1956},
        ISSN={0025-5521},
     journal={Math. Scand.},
      volume={4},
       pages={287\ndash 302},
         url={https://doi.org/10.7146/math.scand.a-10478},
      review={\MR{0085223}},
}

\bib{height}{incollection}{
      author={Shaska, T.},
      author={Beshaj, L.},
       title={Heights on algebraic curves},
        date={2015},
   booktitle={Advances on superelliptic curves and their applications},
      series={NATO Sci. Peace Secur. Ser. D Inf. Commun. Secur.},
      volume={41},
   publisher={IOS, Amsterdam},
       pages={137\ndash 175},
      review={\MR{3525576}},
}

\bib{swa}{article}{
      author={Swan, Richard~G.},
       title={Factorization of polynomials over finite fields},
        date={1962},
        ISSN={0030-8730},
     journal={Pacific J. Math.},
      volume={12},
       pages={1099\ndash 1106},
         url={http://projecteuclid.org/euclid.pjm/1103036322},
      review={\MR{0144891}},
}

\bib{Za}{article}{
      author={Zarhin, Yuri~G.},
       title={Galois groups of {M}ori trinomials and hyperelliptic curves with
  big monodromy},
        date={2016},
        ISSN={2199-675X},
     journal={Eur. J. Math.},
      volume={2},
      number={1},
       pages={360\ndash 381},
         url={https://doi.org/10.1007/s40879-015-0048-2},
      review={\MR{3454107}},
}

\end{biblist}
\end{bibdiv}

\end{document}